\documentclass[amssym,11pt]{article}
\usepackage{graphicx}
\usepackage{geometry}
\usepackage{amsmath}
\usepackage{latexsym}
\usepackage{color}
\usepackage{amsmath, amsthm, amscd, amsfonts, amssymb}
\usepackage{caption}
\newcommand {\smat}      [1] {\left[\begin{smallmatrix}{#1}}
	\newcommand {\srix}          {\end{smallmatrix}\right]}
\def\bkR{{\rm I\kern-.17em R}}
\def\bkC{{\rm ^{_|}\kern-.47em C}}

\theoremstyle{definition}
\newtheorem{definition}{Definition}[section]
\theoremstyle{theorem}
\newtheorem{theorem}[definition]{Theorem}
\newtheorem{lemma}[definition]{Lemma}
\newtheorem{proposition}[definition]{Proposition}

\newtheorem{remark}[definition]{Remark}
\theoremstyle{definition}
\newtheorem{example}[definition]{Example}
\DeclareMathOperator{\diff}{d}

\newcommand{\dx}{{\diff}x}
\newcommand{\ds}{{\diff}s}

\newcommand{\dt}{{\diff}t}
\newcommand{\dW}{{\diff}W}

\newcommand{\expn}{\operatorname{e}}

\newcommand{\vect}{\operatorname{vec}}

\newcommand{\beq}{\begin{equation}}
	\newcommand{\eeq}{\end{equation}}
\newcommand {\mat}      [1] {\left[\begin{array}{#1}}
	\newcommand {\rix}          {\end{array}\right]}

\newcommand {\s}      [1] {\begin{smallmatrix}{#1}}
	\newcommand {\se}          {\end{smallmatrix}}
\newcommand{\trace}{\operatorname{tr}}

\title{(Empirical) Gramian-based dimension reduction for stochastic differential equations driven by fractional Brownian motion}

\author{
	Nahid Jamshidi\thanks{Martin Luther University Halle-Wittenberg, Institute of Mathematics, Theodor-Lieser-Str. 5, 06120 Halle (Saale), Germany (Email: {\tt
			Nahid.Jamshidi@mathematik.uni-halle.de})}
	\and Martin Redmann\thanks{Martin Luther University Halle-Wittenberg, Institute of Mathematics, Theodor-Lieser-Str. 5, 06120 Halle (Saale), Germany (Email: {\tt
			martin.redmann@mathematik.uni-halle.de}); corresponding author}
}

\begin{document}
	
	\maketitle
	
	\begin{abstract}
		In this paper, we investigate large-scale linear systems driven by a fractional Brownian motion (fBm) with Hurst parameter $H\in [1/2, 1)$. We interpret these equations either in the sense of Young ($H>1/2$) or Stratonovich ($H=1/2$). Especially fractional Young differential equations are well suited for modeling real-world phenomena as they capture memory effects, unlike other frameworks. Although it is very complex to solve them in high dimensions, model reduction schemes for Young or Stratonovich settings have not yet been studied much.
		To address this gap, we analyze important features of fundamental solutions associated to the underlying systems. We prove a weak type of semigroup property which is the foundation of studying system Gramians. From the introduced Gramians, dominant subspace can be identified which is shown in this paper as well. The difficulty for fractional drivers with $H>1/2$ is that there is no link of the corresponding Gramians to algebraic equations making the computation very difficult. Therefore, we further propose empirical Gramians that can be learned from simulation data. Subsequently, we introduce projection-based reduced order models using the dominant subspace information. We point out that such projections are not always optimal for Stratonovich equations as stability might not be preserved and since the error might be larger than expected. Therefore, an improved reduced order model is proposed for $H=1/2$. We validate our techniques conducting numerical experiments on some large-scale stochastic differential equations driven by fBm resulting from spatial discretizations of fractional stochastic PDEs. Overall, our study provides useful insights into the applicability and effectiveness of reduced order methods for stochastic systems with fractional noise, which can potentially aid in the development of more efficient computational strategies for practical applications.
	\end{abstract}
	%% maketitle must follow the abstract.
	\maketitle                   % Produces the title.
	\section{Introduction}
	
	Model order reduction (MOR) is an important tool when it comes to solving high-dimensional dynamical systems. MOR is for instance exploited in the optimal control context, where many system evaluations are required and is successfully used in various other applications. There has been an enormous interest in these techniques for deterministic equations. Let us refer  to \cite{morAnt05, benner2017model}, where an overview on different approaches is given and further references can be found. MOR for Ito stochastic differential equations is also very natural thinking of computationally very involved techniques like Mont-Carlo methods. There has been a vast progress in the development for MOR schemes in the Ito setting. Let us refer to \cite{redmannbenner, mliopt, pod_sde} in order to point out three different approaches in this context.\smallskip
	
	In this paper, we focus on MOR for stochastic systems driven by fractional Brownian motions and with non zero initial data. A fractional Brownian motion is an excellent candidate for simulating various phenomena in practice due to its self-similarity and the long-range dependency. However, when $H\neq \frac{1}{2}$, the process $W^H$ is neither a semimartingale nor a Markov process. These are the main obstacles when MOR techniques are designed for such systems. The dimension reduction, we focus on, is conducted by identifying the dominant subspaces using quadratic forms of the solution of the stochastic equation. These matrices are called Gramians. By characterizing the relevance of different state directions using Gramians, less important information can be removed to achieve the desired reduced order model (ROM). Our work considers various types of Gramians depending on the availability in the different settings. Exact Gramians on compact intervals $[0, T]$ as well as on infinite time horizons are studied. These have previously been used in deterministic frameworks \cite{morGawJ90, morMoo81} or Ito stochastic differential equations \cite{beckerhartmann, morBenD11, redmannbenner, redmann2022gramian}. Given the Young case of $H>1/2$, the fractional driver does not have independent increments making it hard to extend the concept of Gramians to this setting. One of our contributions is the analysis of fundamental solutions of Young differential equations. We prove a weak form of semigroup property which is the basis for a proper definition of Gramians for $H>1/2$. Subsequently, we show that certain eigenspaces of these Gramians are associated to dominant subspaces of the system.
	However, this approach is still very challenging from the computational point of view for fractional Brownian motions with $H>1/2$. This is due to a missing link of the here proposed exact Gramians to Lyapunov equations. This is for the reason that the increments of the driver are not independent. Therefore, empirical Gramians based on simulation data are introduced. Computing this approximation of the exact Gramians is still challenging yet vital since they are needed for deriving the ROM. We further point out, how exact Gramians can be computed for Stratonovich stochastic differential equations. Here, the equivalence to Ito equations is exploited. Although we show that these Gramians identify redundant information in Stratonovich settings, MOR turns out to be not as natural as in the Ito case. In fact, we illustrate that projection-based dimension reduction for Stratonovich equations leads to ROMs that lack important properties. For instance, stability might not be preserved in the ROM and the error does not solely depend on the truncated eigenvalues of the Gramians. This indicates that there are situations in which the projection-based ROM performs poorly. For that reason, we propose a modification of the ROM having all these nice properties known for Ito equations (stability and meaningful error bounds).\smallskip
	
	The paper is structured as follows. In Section \ref{sec2}, we provide a quick introduction to fractional Brownian motions as well as the associated integration. In particular, we define Young's integral ($H>1/2$) and the integrals in the sense of Ito/Stratonovich ($H=1/2$). In Section \ref{stochstabgen}, we briefly discuss the setting and the general reduced system structure by projection. We give a first insight on how projection-based reduced systems need to be modified in order to ensure a better approximation quality in the Stratonovich setting. Section \ref{sec_fund_sol-Gram} contains a study of properties of the fundamental solution to the underlying stochastic system. A weak type of semigroup property leads to a natural notion of Gramians. We show that these Gramians indeed characterize dominant subspaces of the system and are hence the basis of our later dimension reduction. Since exact Gramian are not available for each choice of $H$, we discuss several modifications and approximations in this context. Strategies on how to compute Gramians for Stratonovich equations are delivered as well. In Section \ref{sec4}, we describe the concept of balancing for all variations of Gramians that we have proposed. A truncation procedure then yields a ROM. We further point out that the truncation method is not optimal in the Stratonovich case ($H=1/2$) and suggest an alternative that is based on transformation into the equivalent Ito framework. Finally, in Section \ref{sec5}, we utilize the methods described in the previous sections to solve stochastic heat and wave equations with fractional noises. This section presents the results of our simulations and demonstrates the effectiveness of the proposed methods in solving these equations with various noise cases.

	\section{Fractional Brownian motion and Young/Stratonovich integration}\label{sec2}
	
	Below, it is assumed that all stochastic processes occurring in this paper are defined on the filtered probability space $ (\Omega, \mathcal{F},(\mathcal{F}_t)_{t\in[0,T]},\mathbb{P})$. Our main focus is on the fractional Brownian motion (fBm) $W^H(t)$, $t\geq 0$, with Hurst parameter $H\in(0,1)$. It is a Gaussian process with mean zero and covariance function given by
	\begin{align}\label{cov_fbm}
		\mathbb{E}[W^H(t)W^H(s)]=\frac{1}{2}\left(s^{2H}+t^{2H}-|t-s|^{2H}\right).\end{align}
	The fBm was initially proposed by Kolmogorov, and it was later investigated by Mandelbrot and Van Ness, who developed a stochastic integral representation of it using a standard Brownian motion. Additionally, Hurst used statistical analysis of annual Nile river runoffs to create the Hurst index, which is a resulting parameter.
	
	The fBm exhibits self-similarity, which means that the probability distributions of the processes $a^{-H}W^H(at)$, $t\geq 0$, and $W^H(t)$, $t\geq 0$, are the same for any constant $a>0$, which is a direct consequence of the covariance function. The increments of the fBm are stationary and, if $H=1/2$, they are also independent.
	
	The Hölder continuity of the fBm trajectories can be calculated using the modulus of continuity described in \cite{Garsia}. To be more precise, we find a non-negative random variable $G_{\epsilon,T}$ for all $\epsilon>0$ and $T>0$, so that
	\[|W^H(t)-W^H(s)|\leq G_{\epsilon,T}|t-s|^{H-\epsilon},\]
	almost surely, for all $s,t \in [0,T]$.
	Therefore, the Hurst parameter $H$ not only accounts for the correlation of the increments but also characterizes the regularity of the sample paths. In other words, the trajectories are Hölder continuous with parameter arbitrary close to $H$.
	
	In the following, we will always consider $H\geq 1/2$ and briefly recall the corresponding integration theory. In order to cover the ``smooth case'' $H>1/2$, the integral defined by Young \cite{Young} in 1936 is considered. This scenario covers integrands and integrators with a certain Hölder regularity.
	\begin{definition}\label{defn_young}
		Let $ C^{\alpha}  $ be the set of Hölder continuous functions defined on $ [0, T] $, with index $ 0<\alpha \leq 1 $.
		Suppose that $ f \in C^{\alpha} $  and $ g \in C^{\beta}$, where  $  \alpha + \beta > 1 $. Given a sequence $(t_i^n )^{k_n}_{i=0}$  of partitions of $[0, T]$ with $\lim_{n\to\infty}\max_{i=0}^{k_n-1}\{t_{i+1}^n-t_{i}^n\}=0$. Then, the Young integral $ \int_0^T f(s) dg(s) $ is then defined as
		\[ \int_0^T f(s) dg(s):=\lim_{n\rightarrow \infty} \sum_{i=0}^{k_n-1}f(t_i^n)\big[g(t_{i+1}^n)-g(t^n_i)\big].\]
		As the paths of $W^H$ are a.s. Hölder continuous with  $\alpha=H-\epsilon$, we define $\int_0^T Y(s)\circ \dW^H(s)$ for processes $Y$ with $(H-\epsilon)$-Hölder continuous trajectories path-wise in the sense of Young if $H>1/2$.
	\end{definition}
	$H=1/2$ represents the boundary case, in which Young integration does not work anymore. For that reason, the probabilistic approach of Stratonovich is chosen in the following.
	\begin{definition}\label{defn_strat}
		Let $H=1/2$ and $(t_i^n )^{k_n}_{i=0}$ a partition like in Definition \ref{defn_young}. Given a continuous semimartingale $Y$, we set
		\begin{align*}
			\int_0^T Y(s)\circ \dW^H(s)   := \int_0^T Y(s) \dW^H(s) + \frac{1}{2}[Y, W^H]_T,                                                                                                                                                                                                                                                                                                                                                                                                                                                                                                                                                                                                                                                                                                                      \end{align*}
		where the first term is the Ito integral $\int_0^T Y(s) \dW^H(s):=\mathbb P-\lim_{n\rightarrow \infty} \sum_{i=0}^{k_n-1}Y(t_i^n)\big[W^H(t_{i+1}^n)-W^H(t^n_i)\big]$ and $[Y, W^H]_T:=\mathbb P-\lim_{n\rightarrow \infty} \sum_{i=0}^{k_n-1}\big[Y(t_{i+1}^n)-Y(t_i^n)\big]\big[W^H(t_{i+1}^n)-W^H(t^n_i)\big]$ is the quadratic covariation. The expression ``$\mathbb P-\lim$'' indicates the limit in probability.
	\end{definition}
	Let us refer to, e.g., \cite{kloeden_platen, oksendal2013stochastic} for more details concerning stochastic calculus given $H=1/2$.
	The Stratonovich integral can be viewed as the natural extension of Young, since the Stratonovich setting still ensures having a ``classical'' chain rule. Moreover, $W^H$, $H=1/2$, can be approximated by ``smooth'' processes $W^{H, \epsilon}$ with bounded variation paths when Stratonovich stochastic differential equations are considered, e.g.,  $W^{H, \epsilon}$ can be piece-wise linear (Wong-Zakai) \cite{wong_zakai_phd, wong_zakai}. Due to these connections and in order to distinguish from the Ito setting, we use the circle notation $\circ \dW^H$ for both the Young and the Stratonovich case. It is worth mentioning that the lack of martingale property makes the analysis of such integrals particularly challenging, and might require advanced mathematical techniques such as Malliavin calculus, see for instance \cite{Alos}. Nevertheless, Young and Stratonovich differential equations driven by a fBm have important applications in various fields.
	
	\section{Setting and (projection-based) reduced system}\label{stochstabgen}
	We consider the following Young/Stratonovich stochastic differential equation controlled by $u$ satisfying $\|u\|_{L^2_T}^2:= \mathbb E\int_0^T \|u(t)\|_2^2 \dt<\infty$:
	\begin{equation}\label{goal_a}
		\begin{split}
			\dx(t) &=[Ax(t)+Bu(t)]\dt+\sum_{i=1}^qN _i x(t)\circ\dW_{i}^H(t), \quad x(0)=x_0=X_0z,\\
			y(t)&=Cx(t),\quad t\in[0,T],
		\end{split}
	\end{equation}
	where $ A,N_i\in \mathbb{R}^{n\times n} $, $ B\in\mathbb{R}^{n\times m} $, $ C\in \mathbb{R}^{p\times n}  $, $X_0\in \mathbb{R}^{n\times v} $, $z\in\mathbb{R}^v$ and $T>0$ is the terminal time. $W_1^H,\dots,W_q^H $ are independent fBm with Hurst index $H\in[1/2,1)$.
	System \eqref{goal_a} is defined as an integral equation using Definitions \ref{defn_young} ($H>1/2$) and \ref{defn_strat} ($H=1/2$)  to make sense of $\int_0^t N_ix(s)\circ\dW_i^H(s)$.
	
	For the later dimension reduction procedure, it can be beneficial to rewrite the Stratonovich setting in the Ito form. Given $H=1/2$, the state equation in \eqref{goal_a} is equivalent to the Ito equation
	\begin{equation}\label{stratonovich}
		\dx(t) =[(A+\frac{1}{2}\sum_{i=1}^q N_i^2)x(t)+Bu(t)]\dt+\sum_{i=1}^q N _i x(t)\dW_{i}^H(t)
	\end{equation}
	exploiting that the quadratic covariation process is $\sum_{i=1}^q\int_0^t N_i^2 x(s)ds$, $t\in [0, T]$.
	
	The goal of this paper is to find a system of reduced order. This can be done using projection methods, in which a subspace spanned by the columns of $ V\in \mathbb{R}^{n\times r} $ is identified, so that  $x(t)\approx Vx_r(t)$. Inserting this into \eqref{goal_a} yields
	\begin{equation}\label{er}
		Vx_r(t) = X_0z+ \int_0^t [AV x_r(s)+Bu(s)]\ds+\sum_{i=1}^q \int_0^t N_iV x_r(s)\circ\dW^H_i(s)+e(t),\quad y_r(t) =CVx_r(t),
	\end{equation}
	We enforce the error $e(t)$ to be orthogonal to some space spanned by columns of $W\in\mathbb{R}^{n\times r}$, for which we assume that $ W^\top V=I$. Multiplying \eqref{er} with $W^\top$ from the left yields
	\begin{equation}\label{reduced}
		\begin{split}
			dx_r(t) &=[A_r x_r(t)+B_r u(t)]\dt+\sum_{i=1}^q N_{i,r} x_r(t)\circ\dW_{i}^H(t), \quad x_r(0)=x_{0,r}= X_{0, r} z,\\
			y_r(t)&=C_r x_r(t),\quad t\in[0,T],
		\end{split}
	\end{equation}
	where $X_{0, r}= W^\top X_0$ and
	\[A_r=W^\top AV,\quad B_r=W^\top B,\quad N_{i,r}=W^\top N_iV,\quad C_r=CV.\]
	This type of approximation can be interpreted as a Petrov-Galerkin projection. If $W=V$ has orthonormal columns, we obtain a Galerkin approximation. On the other hand, we want to point out that reduced order systems can also be of a different form when $H=1/2$. Inserting $x(t)\approx Vx_r(t)$ into \eqref{stratonovich} instead of \eqref{goal_a} and conducting the same Petrov-Galerkin procedure, we obtain a reduced Ito system with drift coefficient $A_r+ \frac{1}{2}\sum_{i=1}^q W^\top N_i^2 V$. Transforming this back into a Stratonovich equation yields \begin{align}\label{alternative_reduced}
		d\bar x_r(t) =[\big(A_r+\frac{1}{2}\sum_{i=1}^q (W^\top N_i^2 V- N_{i,r}^2)\big)\bar x_r(t)+B_r u(t)]\dt+\sum_{i=1}^q N_{i,r} \bar x_r(t)\circ\dW_{i}^H(t),                                                                                                                                                                                                                                                                                                                                                                                                                                                                                                                                                                                                                                                                                                                                                                                                                                                                                                                 \end{align}
	which is clearly different from the state equation in \eqref{reduced}. This is due to the Ito-Stratonovich correction not being a linear transformation. Another goal of this paper is to analyze whether $x_r$ or $\bar x_r$ performs better for $H=1/2$.
	
	\section{Fundamental solutions and Gramians}\label{sec_fund_sol-Gram}
	
	\subsection{Fundamental solutions and their properties}
	
	Before we are able to compute suitable reduced systems, we require fundamental solutions $\Phi$. These $\Phi$ will later lead to the concept of Gramians that identify dominant subspaces.
	The fundamental solution associated to \eqref{goal_a} is a two parameter matrix valued stochastic process $\Phi$ solving
	\begin{equation}\label{fun_sol}
		\Phi(t,s)=I+\int_s^tA\Phi(\tau,s)d\tau+\sum_{i=1}^q\int_s^tN_i\Phi(\tau,s)\circ\dW^H_i(\tau)
	\end{equation}
	for $ t\geq s\geq 0 $. For simplicity, we set $\Phi(t):=\Phi(t,0)$ meaning that we omit the second argument if it is zero. We can separate the variables, since we have $ \Phi(t,s)=\Phi(t)\Phi(s)^{-1} $ for $ t\geq s\geq 0 $. This is due to the fact that $ \Phi(t)\Phi^{-1}(s) $ fulfills equation \eqref{fun_sol}. Now, we derive the solution of the state equation \eqref{goal_a} in the following proposition.
	\begin{proposition}\label{Phi_inv}
		The solution of the state equation \eqref{goal_a} for $ H\in[1/2, 1)$ is given by
		\begin{equation}\label{Ex_sol}
			x(t)=\Phi(t)x_0+\int_0^t \Phi(t, s) B u(s)\ds,\quad t\in [0, T].\end{equation}
	\end{proposition}
	\begin{proof}
		Defining $k(t)=x_0+\int_0^t \Phi(s)^{-1} B u(s)\ds$, the result is an immediate consequence of applying the classical product rule (available in the Young/Stratonovich case) to
		$ \Phi(t)k(t), t\in [0, T]$. It follows that
		\begin{align*}    \Phi(t)k(t)&=x_0+\int_0^t \Phi(s)dk(s)+\int_0^t (d\Phi(s)) k(s)\\
			&=x_0+\int_0^t Bu(s)\ds+ \int_0^tA\Phi(s)k(s)\ds+ \sum_{i=1}^q\int_0^t N_i\Phi(s)k(s)\circ \dW^H_i(s),
		\end{align*}
		meaning that $\Phi(t)k(t)$, $t\in [0, T]$, is the solution to \eqref{goal_a}. The result follows by $\Phi(t,s)=\Phi(t)\Phi^{-1}(s)$.
	\end{proof}
	The fundamental solution lacks the strong semigroup feature compared to the deterministic case $ (N_i = 0)$. This means that,  $ \Phi(t,s) = \Phi(t-s) $ does not hold $\mathbb P$-almost surely, as the trajectories of $W^H$ on $ [0,t-s] $ and $ [s,t] $ are distinct. In the following lemma, we can demonstrate that the semigroup property holds in distribution exploiting the stationary increments of $W^H$.
	\begin{lemma}\label{Phi}
		It holds that the fundamental solution of \eqref{goal_a} satisfies
		\[\Phi(t,s)\overset{d}{=} \Phi(t-s), \quad t\geq s\geq 0.\]
	\end{lemma}
	\begin{proof}
		We consider $\Phi(\cdot)$ on the interval $[0, t-s]$ and $\Phi(\cdot,s)$ on $[s, t]$. Introducing the step size $\Delta t = \frac{t-s}{N}$, we find the partitions $t_k= k \Delta t$ and $t_k^{(s)}=s+t_k$, $k\in\{0, 1, \dots, N\}$, of $[0, t-s]$ and $[s, t]$.
		We introduce the Euler discretization of \eqref{fun_sol} as follows
		\begin{equation}\label{euler_mothod_young}
			\begin{split}
				\Phi_{k+1}&=\Phi_k+A\Phi_k\Delta t+\sum_{j=1}^q N_j\Phi_k\Delta W^{H}_{j, k},\\
				\Phi_{k+1}^{(s)}&=\Phi_k^{(s)}+A\Phi_k^{(s)}\Delta t+\sum_{j=1}^q N_j\Phi_k^{(s)}\Delta W^{H, (s)}_{j, k},
			\end{split}
		\end{equation}
		where we define $\Delta W^{H}_{j,k}=W^H_{j}(t_{k+1})-W^H_{j}(t_k)$ and $\Delta W^{H, (s)}_{j, k}=W^H_{j}(t_{k+1}^{(s)})-W^H_{j}(t_k^{(s)})$. According to \cite{Mishura, Neuenkirch}, the Euler scheme converges $\mathbb P$-almost surely for $H>1/2$ yielding in particular convergence in distribution, that is
		\begin{equation}\label{distribution}
			\Phi_N\xrightarrow{\enskip d\enskip} \Phi(t-s) ,\quad \Phi_N^{(s)}\xrightarrow{\enskip d \enskip} \Phi(t,s),
		\end{equation}
		as $N\rightarrow \infty$. The Euler method does not converge almost surely in the Stratonovich setting. However, for $H=1/2$, we can rewrite \eqref{fun_sol} as the Ito equation $
		\Phi(t,s)=I+\int_s^t(A+\frac{1}{2}\sum_{i=1}^q N_i^2)\Phi(\tau,s)d\tau+\sum_{i=1}^q\int_s^tN_i\Phi(\tau,s)\dW^H_i(\tau)$. This equation can be discretized by a scheme like in \eqref{euler_mothod_young} (Euler-Maruyama). The corresponding convergence is in $L^1(\Omega, \mathcal F, \mathbb P)$ \cite{kloeden_platen}, so that we also have \eqref{distribution} for $H=1/2$ as well.
		By simple calculation we can get from \eqref{euler_mothod_young} that
		\begin{equation*}
			\begin{split}
				\Phi_{N}&=\prod_{k=0}^{N-1}\left( I+A\Delta t+\sum_{j=1}^qN_j\Delta W^{H}_{j, k}\right)=:F(Z),\\
				\Phi_{N}^{(s)}&=\prod_{k=0}^{N-1}\left(I+A\Delta t+\sum_{j=1}^q N_j\Delta W^{H, (s)}_{j, k}\right)=F(Z^{(s)}),
			\end{split}
		\end{equation*}
		where $Z:=(\Delta W^{H}_{j, k})$ and $Z^{(s)}:=(\Delta W^{H, (s)}_{j, k})$ ($j=1, \dots, q$ and $k=0, \dots, N-1$) are Gaussian vectors with mean zero. Notice that the function $F$ is just slightly different for $H=1/2$, i.e., $A$ is replaced by $A+\frac{1}{2}\sum_{i=1}^q N_i^2$. It remains to show that the covariance matrices of $Z$ and $Z^{(s)}$ coincide leading to $\Phi_N(t,s)\overset{d}{=} \Phi_N(t-s)$. Subsequently, the result follows by \eqref{distribution}.
		Using the independence of $W^H_i$ and $W^H_j$ for $i\neq j$, the non zero entries of the covariances of $Z$ and $Z^{(s)}$ are
		$\mathbb{E}[\Delta W^{H}_{j, k} \Delta W^{H}_{j, \ell}]$ and $\mathbb{E}[\Delta W^{H, (s)}_{j, k} \Delta W^{H, (s)}_{j, \ell}]$  ($k, \ell = 0, 1, \dots, N-1$), respectively. These expressions are the same, since exploiting \eqref{cov_fbm}, we obtain that
		\begin{align*}
			\mathbb{E}[\Delta W^{H, (s)}_{j, k} \Delta W^{H, (s)}_{j, \ell}]&=\mathbb{E}[\big(W^{H}_{j}(s+t_{k+1})-W^{H}_{j}(s+t_k)\big)\big(W^{H}_{j}(s+t_{\ell+1})-W^{H}_{j}(s+t_\ell)\big)]\\
			&=\frac{1}{2}\left(\vert t_{k+1}-t_\ell \vert^{2 H}+\vert t_{k}-t_{\ell+1} \vert^{2 H}-\vert t_{k+1}-t_{\ell+1} \vert^{2 H}-\vert t_{k}-t_\ell\vert^{2 H}\right)
		\end{align*}
		is independent of $s$. This concludes the proof.
	\end{proof}
	
	\subsection{Exact and empirical Gramians}\label{sec_Gram}

	\subsubsection{Exact Gramians and dominant subspaces}
	
	Similar to the approach presented in the second POD-based method outlined in the reference \cite{Jamshidi}, our methodology involves partitioning the primary system described in equation \eqref{goal_a} into distinct subsystems in the following manner:
	\begin{align}\label{stan1}
		dx_u(t)&=[Ax_u(t)+Bu(t)]\dt+\sum_{i=1}^qN_i x_u(t) \circ dW_i^H(t), \quad x_u(0)=0,\quad
		y_u(t)=Cx_u(t), \\
		\label{stan2}
		dx_{x_0}(t)&= Ax_{x_0}(t)\dt+\sum_{i=1}^q N_i x_{x_0}(t)\circ dW_i^H(t), \quad x_{x_0}(0)=x_0=X_0z,\quad
		y_{x_0}(t)=Cx_{x_0}(t).
	\end{align}
	Proposition \ref{Phi_inv} shows that we have the representations $x_{x_0}(t)=\Phi(t)x_0$ and $x_u(t)=\int_0^t \Phi(t, s) B u(s)\ds$, so that $y(t)=y_{x_0}(t)+ y_u(t)$ follows. Lemma \ref{Phi} is now vital for a suitable definition of Gramians. Due to the weak semigroup property of the fundamental solution in Lemma \ref{Phi}, it turns out that
	\begin{equation}\label{subsys_gram}
		P_{u, T}:=\mathbb{E}\bigg[ \int_0^T\Phi(s)BB^\top\Phi(s)^\top\ds \bigg],\quad
		P_{x_0, T}:=\mathbb{E}\bigg[ \int_0^T\Phi(s)X_0X_0^\top\Phi(s)^\top\ds \bigg].
	\end{equation}
	are the right notion of Gramians for \eqref{stan1} and \eqref{stan2}. With \eqref{subsys_gram} we then define a Gramian $P_T:=P_{u, T}+P_{x_0, T}$ for the original state equation \eqref{goal_a}. In case of the output equation in \eqref{goal_a}, a Gramian can be introduced directly by
	\begin{align*}
		Q_{T}:=\mathbb E\int_0^{T} \Phi(s)^\top C^\top C \Phi(s) ds.
	\end{align*}
	\begin{proposition}\label{dominant_subspace_char}
		Given $v \in \mathbb R^n$, an initial state of the form $x_0=X_0 z$ and a deterministic control $u\in L^2_T$, then we have that \begin{align}\label{subsys_dom_space}
			\int_0^T\mathbb{E}\langle x_{x_0}(t), v\rangle^2_2\dt \leq  v^\top P_{x_0, T} v  \|z\|^2_2, \quad \sup_{t\in[0,T]}\mathbb{E}\vert\langle x_u(t), v\rangle_2\vert^2 \leq v^\top P_{u, T} v \|u\|_{L^2_T}^2
		\end{align}
		and consequently\begin{align}\label{full_state_dom}
			\int_0^T\mathbb{E}\langle x(t), v\rangle^2_2\dt \leq 2 v^\top P_{T} v \max\{\|z\|^2_2,T \|u\|_{L^2_T}^2\}.
		\end{align}
		Moreover, it holds that \begin{equation}\label{mu}
			\int_0^T \mathbb{E}\|C\Phi(t)v\|_2^2\dt=v^\top Q_T v.
		\end{equation}
	\end{proposition}
	\begin{proof}
		The first relation is a simple consequence of the inequality of Cauchy-Schwarz and the representation of $x_{x_0}$ in Proposition \ref{Phi_inv}. Thus,
		\begin{align*}
			\int_0^T\mathbb{E}\langle x_{x_0}(t), v\rangle^2_2\dt&=\mathbb{E}\int_0^T\langle \Phi(t)X_0z, v\rangle^2_2\dt = \mathbb{E}\int_0^T\langle z, X_0^\top\Phi(t)^\top v\rangle_2^2\dt\leq  \|z\|^2_2 \mathbb E\int_0^T\| X_0^\top\Phi(t)^\top v\|^2_2\dt\\
			&=v^\top P_{x_0, T} v  \|z\|^2_2.
		\end{align*}
		Utilizing equation \eqref{Ex_sol} and the Cauchy-Schwarz inequality once more, we have
		\begin{equation*}
			\begin{split}
				\mathbb{E}\langle x_u(t), v\rangle^2_2&=\mathbb{E}\langle \int_0^t\Phi(t,s)Bu(s)\ds, v\rangle^2_2=\mathbb{E}\left[\left(\int_0^t\langle \Phi(t,s)Bu(s), v\rangle_2\ds\right)^2\right]\\&\leq \mathbb{E}\left[\left(\int_0^t\langle u(s), B^\top\Phi(t,s)^\top v\rangle_2\ds\right)^2\right] \leq v^\top\mathbb{E}\int_0^t\Phi(t,s)BB^\top\Phi(t,s)^\top\ds\, v \,\|u\|_{L^2_t}^2.
			\end{split}
		\end{equation*}
		Based on Lemma \ref{Phi}, we obtain that $\mathbb E\left[\Phi(t,s)BB^\top\Phi(t,s)^\top\right]=\mathbb E\left[\Phi(t-s)BB^\top\Phi(t-s)^\top\right]$. Hence, \begin{equation*}
			\begin{split}
				\mathbb{E}\langle x_u(t), v\rangle^2_2&\leq v^\top\mathbb{E}\int_0^t\Phi(t-s)BB^\top\Phi(t-s)^\top\ds\, v \,\|u\|_{L^2_t}^2\leq v^\top P_{u, T}\, v \,\|u\|_{L^2_T}^2
			\end{split}
		\end{equation*}
		by variable substitution and the increasing nature of $P_{u, T}$  and $\|u\|_{L^2_T}^2$ in $T$. This shows the second part of \eqref {subsys_dom_space}. Exploiting Proposition \ref{Phi_inv}, we know that $x= x_{x_0} + x_u$. Therefore, we have \begin{align*}
			\int_0^T\mathbb{E}\langle x(t), v\rangle^2_2\dt &\leq 2\Big(\int_0^T\mathbb{E}\langle x_{x_0}(t), v\rangle^2_2\dt+\int_0^T\mathbb{E}\langle x_{u}(t), v\rangle^2_2\dt\Big)\\
			&\leq 2\Big(\int_0^T\mathbb{E}\langle x_{x_0}(t), v\rangle^2_2\dt+T\sup_{t\in[0,T]}\mathbb{E}\langle x_{u}(t), v\rangle^2_2\Big)
		\end{align*}
		by the linearity of the inner product in the first argument. Applying \eqref{subsys_dom_space} to this inequality yields \eqref{full_state_dom} using that $P_T=P_{x_0, T}+ P_{u, T}$. By the definitions of $Q_T$ and the Euclidean norm, \eqref{mu} immediately follows, so that this proof is concluded.
	\end{proof}
	\begin{remark}
		If the limits $P_{x_0}=\lim_{T\rightarrow \infty} P_{x_0, T}$, $P_u=\lim_{T\rightarrow \infty} P_{u, T}$, $P=\lim_{T\rightarrow \infty} P_T$ and $Q=\lim_{T\rightarrow \infty} Q_T$ exist, the Gramians in Proposition \ref{dominant_subspace_char} can be replaced by their limit as we have $v^\top P_T v \leq v^\top P v$, $v^\top Q_T v \leq v^\top Q v$ etc for all $v\in \mathbb R^n$.
	\end{remark}
	\begin{remark}\label{interprete_prop}
		We can read Proposition \ref{dominant_subspace_char} as follows. If $v$ is an eigenvector of $P_{x_0, T}$ and $P_{u, T}$, respectively, associated to a small eigenvalue, then $x_{x_0}$ and $x_u$ are small in the direction of $v$. Such state directions can therefore by neglected. The same interpretation holds for $x$ using \eqref{full_state_dom} when $v$ is a respective eigenvector of $P_T$. Now, expanding the initial state as
		\[x_0=\sum_{k=1}^n\langle x_0,q_k\rangle_2q_k,\]
		where $ (q_k)_{k=1,\dots,n} $ represents an orthonormal set of eigenvectors of $ Q_T $, and using the solution representation in \eqref{Ex_sol}, we obtain
		\begin{equation}\label{repr_output}
			y(t)=C\Phi(t)x_0+C\int_0^t\Phi(t,s)Bu(s)\ds
			=\sum_{k=1}^n\langle    x_0,q_k\rangle_2C\Phi(t)q_k+C\int_0^t\Phi(t,s)Bu(s)\ds,
		\end{equation}
		with $ t\in[0,T]$. Identity \eqref{mu} therefore tells us that $v=q_k$ contributes very little to $y$ if the corresponding eigenvalue is small. Such $q_k$ can be removed from the dynamics without causing a large error in \eqref{repr_output}.
	\end{remark}
	\subsubsection{Approximation and computation of Gramians}\label{sec_comp_gram}
	
	In theory, Proposition \ref{dominant_subspace_char} together with Remark \ref{interprete_prop} is the key when aiming to identify dominant subspaces of \eqref{goal_a} that lead to ROMs. However, the Gramians that we defined above are hard to compute. In fact, there is no known link of these Gramians to algebraic Lyapunov equations or matrix differential equations when $H>1/2$. For that reason, we suggest an empirical approach in the following in which approximate Gramians based on sampling are calculated.
	In particular, we consider a discretization of the integral representations by a Monte-Carlo method. Let us introduce a equidistant time grid $0=s_0<s_1<\dots< s_{\mathcal N}=T$ and let $\mathcal N_s$ further be the number of Monte-Carlo samples. Given that $\mathcal N$ and $\mathcal N_s$ are sufficiently large, we obtain
	\begin{equation}\label{empirical_sub_gram}
		\begin{split}
			&{P}_{u, T}\approx\bar {P}_{u, T}=\frac{T}{\mathcal N \cdot \mathcal N_s}\sum_{i=1}^{\mathcal N}\sum_{j=1}^{\mathcal N_s}\Phi(s_i,\omega_j)B B^\top \Phi(s_i,\omega_j)^\top,\\
			&{P}_{x_0, T}\approx\bar{P}_{x_0, T}=\frac{T}{\mathcal N \cdot \mathcal N_s}\sum_{i=1}^{\mathcal N}\sum_{j=1}^{\mathcal N_s} \Phi(s_i,\omega_j)X_0 X_0^\top \Phi(s_i,\omega_j)^\top,
		\end{split}
	\end{equation}
	where $\omega_j\in \Omega$. Now, the advantage is that $\Phi(\cdot)B$ and $\Phi(\cdot) X_0$ are easy to sample as they are the solutions of the control independent variable $x_{x_0}$ in \eqref{stan2} with initial states $x_0\mapsto B$ and $x_0\mapsto X_0$, respectively. This is particularly feasible if $B$ and $X_0$ only have a few columns. Based on \eqref{empirical_sub_gram}, we can then define $\bar P_T:=\bar {P}_{x_0, T}+\bar {P}_{u, T}$ approximating $P_T$. Here, the goal is to choose $\mathcal N$ and $\mathcal N_s$ so that the estimates in Proposition  \ref{dominant_subspace_char} still hold (approximately) ensuring the dominant subspace characterization by the empirical Gramians. Notice that if the limits of the Gramians as $T\rightarrow \infty$ shall be considered, then the terminal time needs to be chosen sufficiently large. In fact, it is also not an issue to write down the empirical version of $Q_T$ which is \begin{align*}                                                                                                                                                                                                                                                                                                                                                                                                                                                                                                                                                                                                                                                                                                                                                                                                                                                                                                                                           \bar {Q}_{T}=\frac{T}{\mathcal N \cdot \mathcal N_s}\sum_{i=1}^{\mathcal N}\sum_{j=1}^{\mathcal N_s}\Phi(s_i,\omega_j)^\top C^\top C \Phi(s_i,\omega_j).                                                                                                                                                                                                                                                                                                                                                                                                                                                                                                                                                                                                                                                                                                                                                                                                                                                                                                                                      \end{align*}
	However, this object is computationally much more involved. This is because $C \Phi(\cdot)$ cannot be linked to an equations that can be sampled easily. In fact, we might have to sample from \eqref{fun_sol}, which is of order $n^2$. This leaves the open question of whether $\bar Q_T$ is numerically tractable. Let us briefly discuss that the computation of $P_T$, $Q_T$ or their limits as $T\rightarrow \infty$ is easier when we are in the Stratonovich setting of $H=1/2$. Once more let us point out the relation between Ito and Stratonovich differential equations. So, the fundamental solution of the state equation in \eqref{goal_a} defined in \eqref{fun_sol} is also the fundamental solution of \eqref{stratonovich}, i.e., it satisfies
	$\Phi(t)=I+\int_0^t A_N\Phi(s)ds+\sum_{i=1}^q\int_0^tN_i\Phi(s)\dW^H_i(s)$, where $A_N:= A+\frac{1}{2}\sum_{i=1}^q N_i^2$. Now, defining the linear operators $\mathcal{L}_{A_N}(X) =A_NX + XA_N^\top $ and $\Pi(X) =\sum_{i=1}^qN_iXN_i^\top$, it is a well-known fact (consequence of Ito's product rule in \cite{oksendal2013stochastic}) that $Z(t)= \mathbb E\big[\Phi(t) M \Phi(t)^\top\big]$ solves \begin{align}\label{lyap_ODE}
		\frac{d}{dt} Z(t)=\mathcal L_{A_N}\big[Z(t)\big]+\Pi\big[Z(t)\big], \quad Z(0) =M,\quad t\geq 0,                                                                                                                                                                                                                                    \end{align}
	where $M$ is a matrix of suitable dimension. We refer, e.g., to \cite{redmann2022gramian} for more details. Setting $M=B B^\top+X_0 X_0^\top$ and integrating \eqref{lyap_ODE} yields \begin{align}\label{eq_PT}
		Z(T)-B B^\top- X_0 X_0^\top=\mathcal L_{A_N}\big[P_T\big]+\Pi\big[P_T\big]                                                                                                                                                                                                                                                                                                                                               \end{align}                                                                                                                                                                using that $P_T=\mathbb{E}\Big[ \int_0^T\Phi(s)\Big(B B^\top+X_0X_0^\top\Big)\Phi(s)^\top\ds \Big]$. If system \eqref{goal_a} is mean square asymptotically stable, that is, $\mathbb E \|\Phi(t)\|^2$ decays exponentially to zero, then we even find
	$-B B^\top- X_0 X_0^\top=\mathcal L_{A_N}\big[P\big]+\Pi\big[P\big]$
	for the limit $P$ of $P_T$. There is still a small gap in the theory left in \cite[Proposition 2.2]{redmann2022gramian} on how to compute $Q_T$ in the case of $H=1/2$. Therefore, the following proposition was stated under the additional assumption that $C^\top C$ is contained in the eigenspace of $\mathcal L_{A_N}^*+ \Pi^*$, where $
	\mathcal L_{A_N}^*(X)= A_N^\top X+X A_N, \quad \Pi^*(X)= \sum_{i=1}^q N_i^\top X N_i$. We prove this result in full generality below.
	\begin{proposition}
		Given that we are in the Stratonovich setting of $H=1/2$. Then, the function $Z_*(t)= \mathbb E\big[\Phi(t)^\top C^\top C \Phi(t)\big]$ solves \begin{align}\label{dual_lyap}
			\frac{d}{dt} Z_*(t)=\mathcal L_{A_N}^*\big[Z_*(t)\big]+\Pi^*\big[Z_*(t)\big], \quad Z_*(0) =C^\top C,\quad t\geq 0.                                                                                                                                                                                                                                     \end{align}
	\end{proposition}
	\begin{proof}
		Let us vectorize the matrix differential equation \eqref{lyap_ODE} leading to $\frac{d}{dt} \vect[Z(t)] = \mathcal K \vect[Z(t)]$, $\vect[Z(0)]=\vect[M]$, where
		\begin{equation*}
			\mathcal{K} = A_N\otimes I+I\otimes A_N+\sum_{i=1}^q N_i\otimes N_i
		\end{equation*}
		with $\otimes$ representing the Kronecker product between two matrices and $\vect[\cdot]$ being the vectorization operator. Therefore, we know that $\expn^{\mathcal K t} \vect[M]=\vect[Z(t)]=\vect\Big[ \mathbb E\big[\Phi(t) M \Phi(t)^\top \big]\Big]=\mathbb E\big[\Phi(t) \otimes \Phi(t) \big]\vect[M]$ again exploiting the relation between the vectorization and the Kronecker product. Since this identity is true for all matrices $M$, we have $\mathbb E\big[\Phi(t) \otimes \Phi(t) \big]=\expn^{\mathcal K t}$. This is now applied to $\vect\big[Z_*(t)\big]= \vect \Big[\mathbb E\big[\Phi(t)^\top C^\top C \Phi(t)\big]\Big]=\mathbb E\big[\Phi(t)^\top \otimes \Phi(t)^\top \big]\vect[C^\top C]= \expn^{\mathcal K^\top t}\vect[C^\top C]$, since $\mathbb E\big[\Phi(t)^\top \otimes \Phi(t)^\top \big]=\left(\mathbb E\big[\Phi(t)\otimes \Phi(t)\big]\right)^\top$. Therefore, it holds that $\frac{d}{dt} \vect[Z_*(t)] = \mathcal K^\top \vect[Z_*(t)]$, $\vect[Z_*(0)]=\vect[C^\top C]$. Devectorizing this equation and exploiting that $\mathcal K^\top$ is the matrix representation of $\mathcal L_{A_N}^*+\Pi^*$ leads to the claim of this proposition.
	\end{proof}
	Integrating \eqref{dual_lyap} and using that $Q_T=\mathbb E\big[\int_0^T\Phi(t)^\top C^\top C \Phi(t)\dt\big]$ leads to
	\begin{align}\label{eq_QT}
		Z_*(T)-C^\top C=\mathcal L_{A_N}^*\big[Q_T\big]+\Pi^*\big[Q_T\big].
	\end{align}
	Once more, mean square asymptotic stability yields the well-known relation $-C^\top C=\mathcal L_{A_N}^*\big[Q\big]+\Pi^*\big[Q\big]$ by taking the limit as $T\rightarrow \infty$ in \eqref{eq_QT}. Although we found algebraic equation \eqref{eq_PT} and \eqref{eq_QT} from which $P_T$ and $Q_T$ could be computed, it is still very challenging to solve these equations. This is mainly due to the unknowns $Z(T)$ and $Z_*(T)$. In fact, \cite{redmann2022gramian} suggests sampling and variance reduction-based strategies to solve \eqref{eq_PT} and \eqref{eq_QT}. We refer to this paper for more details.

	\section{Model reduction of Young/Stratonovich differential equations}\label{sec4}
	In this section, we introduce ROMs that are based on the (empirical) Gramians of Section \ref{sec_Gram} as they (approximately) identify the dominant subspaces of \eqref{goal_a}. In order to accomplish this, we discuss state space transformations first that diagonalize these Gramians. This diagonalization allows to assign unimportant direction in the dynamics to certain state components according to Proposition \ref{dominant_subspace_char}. Subsequently, the issue is split up into two parts. A truncation procedure is briefly explained for the general case of $H\in[1/2, 1)$, in which unimportant state variables are removed. This strategy is associated to (Petrov-)Galerkin schemes sketched in Section \ref{stochstabgen}. Later, we focus on the case of $H=1/2$ and point out an alternative ansatz that is supposed to perform better than the previously discussed projection method.
	Let us notice once more that since a fractional Brownian motion with $H>1/2$ does not have independent increments, no Lyapunov equations associated with the Gramians can be derived. Therefore, we frequently refer to the empirical versions of these Gramians and the corresponding reduced dimension techniques.
	
	\subsection{State space transformation and balancing}
	
	We introduce a new variable $\tilde x(t) = S x(t)$, where $S$ is a regular matrix. This can be interpreted as a coordinate transform that is chosen in order to diagonalize the Gramians of Section \ref{sec_Gram}. This transformation is the basis for the dimension reduction discussed in Sections \ref{mor_sec1} and \ref{mor_sec2}. Now, inserting $\tilde x(t) = S^{-1} x(t)$ into \eqref{goal_a}, we obtain \begin{equation}\label{goal_a_trans}
		\begin{split}
			d\tilde x(t) &=[\tilde A\tilde x(t)+\tilde Bu(t)]\dt+\sum_{i=1}^q \tilde N _i \tilde x(t)\circ\dW_{i}^H(t), \quad \tilde x(0)=\tilde x_0=\tilde X_0z,\\
			y(t)&=\tilde C\tilde x(t),\quad t\in[0,T],
		\end{split}
	\end{equation}
	where $\tilde A=S{A}S^{-1}$, $\tilde B= SB$, $\tilde N_i=S{N_i}S^{-1}$, $\tilde X_0= SX_0$ and $\tilde C= C S^{-1}$. As we can observe from \eqref{goal_a_trans}, the output remains unchanged under the transformation. However, the fundamental solution of the state equation in \eqref{goal_a_trans} is \begin{align}\label{trans_fund}
		\tilde \Phi(t) = S \Phi(t) S^{-1}.                                                                                                                                                                                                                                                                                                                                                                                                                                                                                                                                              \end{align}
	This is obtained by multiplying \eqref{fun_sol} with $S$ from the left and with $S^{-1}$ from the right. Relation \eqref{trans_fund} immediately transfers to the Gramians which are  \begin{align}\label{balanced_gramP}
		\tilde P_T:&= \mathbb E \int_0^T \tilde \Phi(s)(\tilde B\tilde B^\top+\tilde X_0\tilde X_0^\top) \tilde\Phi(s)^\top \ds = S P_T S^\top\\ \label{balanced_gramQ}
		\tilde Q_T:&= \mathbb E \int_0^T \tilde \Phi(s)^\top\tilde C^\top\tilde C \tilde\Phi(s) \ds = S^{-\top} Q_T S^{-1}.
	\end{align}
	Exploiting \eqref{trans_fund} again, the same relations like in \eqref{balanced_gramP} and \eqref{balanced_gramQ} hold true if $P_T$ and $Q_T$ are replaced by their limits $P, Q$ or their empirical versions $\bar P_T, \bar Q_T$. In the next definition, different diagonalizing transformations $S$ are introduced.
	\begin{definition}\label{defn_balancing_pro}
		
		\begin{itemize}
			\item[($i$)] Let the state space transformation $S$ be given by the eigenvalue decomposition $P_T= S^\top \Sigma S$, where $\Sigma$ is the diagonal matrix of eigenvalues of $P_T$. Then, the procedure is called $P_T$-balancing.
			\item[($ii$)] Let $P_T$ and $Q_T$ be positive definite matrices. If $S$ is of the form $  S=\Sigma^{\frac{1}{2}} U^\top L_P^{-1}$
			with the factorization $P_T=L_PL_P^\top$ and the spectral decomposition $L_P^\top Q_T L_P=U\Sigma^2 U^\top$, where $\Sigma^2$ is the diagonal matrix of eigenvalues of $P_TQ_T$. Then, the transformation is called $P_T$/$Q_T$-balancing.
			\item[($iii$)] Replacing $P_T$ and $Q_T$ by their limits (as $T\rightarrow \infty$) in ($i$) and ($ii$), then the schemes are called $P$-balancing or $P$/$Q$-balancing, respectively, where in these cases $\Sigma$ is either the matrix of eigenvalues of $P$ or $\Sigma^2$ contains the eigenvalues of $PQ$.
			\item[($iv$)] Using the empirical versions of $P_T$ and $Q_T$ instead, the methods in ($i$) and ($ii$) are called $\bar P_T$-balancing and $\bar P_T$/$\bar Q_T$-balancing. Here, $\Sigma$ can be viewed as a random diagonal matrix of the respective eigenvalues.
		\end{itemize}
	\end{definition}
	It is not difficult to check that the transformations introduced in Definition \ref{defn_balancing_pro} diagonalize the underlying Gramians. Nevertheless, we formulate the following proposition.
	\begin{proposition}\label{prop_diag_gram}
		\begin{itemize}
			\item Using the matrix $S$ in Definition \ref{defn_balancing_pro} ($i$), we find that the state variable Gramian of system \eqref{goal_a_trans} is $\tilde P_T= \Sigma$.
			\item If instead $S$ is of the form given in Definition \ref{defn_balancing_pro} ($ii$), we have $\tilde P_T=\tilde Q_T= \Sigma$.
			\item The same type of diagonalization is established if the underlying Gramians are either $P, Q$ or $\bar P_T, \bar Q_T$.
		\end{itemize}
	\end{proposition}
	\begin{proof}
		The result follows by inserting the respective $S$ into \eqref{balanced_gramP} and \eqref{balanced_gramQ}. Since these relations also hold true for the pairs $P, Q$ and $\bar P_T, \bar Q_T$, the same argument applies in these cases as well.
	\end{proof}
	Having diagonal Gramians $\Sigma$, Proposition \ref{dominant_subspace_char} (choose $v$ to be the $i$th unit vector in $\mathbb R^n$) together with Remark \ref{interprete_prop}   tells us that we can neglect state components in \eqref{goal_a_trans} that correspond to small diagonal entries $\sigma_i$ of $\Sigma$. Those have to be truncated to obtain a reduced system.

	\subsection{Reduced order models based on projection}\label{mor_sec1}
	
	In that spirit, we decompose the diagonal Gramian based on one of the balancing procedures in Definition \ref{defn_balancing_pro}. We  write \begin{align}\label{partition_sig}
		\Sigma= \begin{bmatrix}{\Sigma}_{1}& \\
			&{\Sigma}_{2}\end{bmatrix},
	\end{align}
	where $\Sigma_1\in\mathbb R^{r\times r}$ contains the $r$ large diagonal entries of $\Sigma$ and $\Sigma_2$ the remaining small ones. We further partition the balanced coefficient of \eqref{goal_a_trans} as follows
	\begin{align}\label{partition}
		\tilde A= \smat{A}_{11}&{A}_{12}\\
		{A}_{21}&{A}_{22}\srix,\quad \tilde{B} = \smat{B}_1\\ {B}_2\srix,\quad \tilde N_i= \smat{N}_{i, 11}&{N}_{i, 12}\\
		{N}_{i, 21}&{N}_{i, 22}\srix\quad \tilde{X}_0 = \smat{X}_{0, 1}\\ {X}_{0, 2}\srix,\quad \tilde{C} = \smat{C}_1 & {C}_2\srix.\end{align}
	The balanced state of \eqref{goal_a_trans} is decomposed as $\tilde x=\smat x_1\\ x_2 \srix$, where $x_1$ and $x_2$ are associated to $\Sigma_1$ and $\Sigma_2$, respectively. Now, exploiting the insights of Proposition \ref{dominant_subspace_char}, $x_2$ barely contributes to \eqref{goal_a_trans}. We remove the equation for $x_2$ from the dynamics and set it equal to zero in the remaining parts. This yields a reduced system \begin{equation}\label{goal_a_truncated}
		\begin{split}
			dx_r(t) &=[A_{11}x_r(t)+ B_1 u(t)]\dt+\sum_{i=1}^q N_{i, 11} x_r(t)\circ\dW_{i}^H(t), \quad x_r(0)=x_{0, r}= X_{0, 1}z,\\
			y_r(t)&=C_1x_r(t),\quad t\in[0,T],
		\end{split}
	\end{equation}
	which is of the form like in \eqref{reduced}. If balancing according to Definition \ref{defn_balancing_pro} is used, then $V$ are the first $r$ columns of $S^{-1}$, whereas $W$ represents the first $r$ columns of $S^\top$. Notice that if solely $P_T$, $P$ or $\bar P_T$ are diagonalized (instead of a pair of Gramians), we have $S^{-1}=S^\top$ and hence $W=V$.

	\subsection{An alternative approach for the Stratonovich setting ($H=1/2$)}\label{mor_sec2}
	
	\subsubsection{The alternative}
	As sketched in Section \ref{stochstabgen}, the truncation/projection procedure is not unique for $H=1/2$ meaning that \eqref{alternative_reduced} can be considered instead of \eqref{goal_a_truncated} (being of the form \eqref{reduced}). Such a reduced system is obtained if we rewrite the state of \eqref{goal_a_trans} as a solution to an Ito equation meaning that $\tilde A$ becomes $\tilde A_N = \tilde A + \frac{1}{2}\sum_{i=1}^q \tilde N_i^2$ in the Ito setting. Now, removing $x_2$ from this system like we explained in Section \ref{mor_sec1}, we obtain a reduced Ito system \begin{equation}\label{alternative_reduced2_ito}
		\begin{split}
			dx_r(t) &=[A_{N, 11}x_r(t)+ B_1 u(t)]\dt+\sum_{i=1}^q N_{i, 11} x_r(t)\dW_{i}^H(t), \quad x_r(0)=x_{0, r}= X_{0, 1}z,\\
			y_r(t)&=C_1x_r(t),\quad t\in[0,T],
		\end{split}
	\end{equation}
	where $A_{N, 11}= A_{11}+\frac{1}{2}\sum_{i=1}^q (N_{i, 11}^2+N_{i, 12}N_{i, 21})$ is the left upper $r\times r$ block of $\tilde A_N$. In Stratonovich form, the system is  \begin{equation}\label{alternative_reduced2_strat}
		\begin{split}
			dx_r(t) &=[(A_{11}+\frac{1}{2}\sum_{i=1}^q N_{i, 12}N_{i, 21})x_r(t)+ B_1 u(t)]\dt+\sum_{i=1}^q N_{i, 11} x_r(t)\circ\dW_{i}^H(t), \quad x_r(0)=x_{0, r}= X_{0, 1}z,\\
			y_r(t)&=C_1x_r(t),\quad t\in[0,T],
		\end{split}
	\end{equation}
	which has a state equation of the structure given in \eqref{alternative_reduced}.
	
	\subsubsection{Comparison of \eqref{goal_a_truncated} and \eqref{alternative_reduced2_strat} for $H=1/2$}\label{drawback_rom}
	
	Let us continue setting $H=1/2$. Moreover, we assume $x_0=0$ in this subsection for simplicity. We only focus on $P$- as well as $P$/$Q$-balancing (explained in Definition \ref{defn_balancing_pro} ($iii$)) in order to emphasize our arguments. In addition, we always suppose that $P$ and $Q$ are positive definite. Let us point out that relations between \eqref{goal_a} and \eqref{alternative_reduced2_strat} are well-studied due to the model reduction theory of Ito equations exploiting that these Stratonovich equations are equivalent to \eqref{stratonovich} and \eqref{alternative_reduced2_ito}. In fact, the (uncontrolled) state equation is mean square asymptotically stable ($\mathbb E \|\Phi(t)\|^2\rightarrow 0$ as $t\rightarrow \infty$) if and only if the same is true for \eqref{stratonovich}. This type of stability is well-investigated in Ito settings, see, e.g.,  \cite{damm, staboriginal}. It is again equivalent to the existence of a positive definite matrix $X$, so that the operator $\mathcal L_{A_N}+\Pi$ evaluated at $X$ is a negative definite matrix, i.e., \begin{align}\label{stab_FOM}
		\mathcal L_{A_N}\big[X\big]+\Pi\big[X\big]  <0.                                                                                                                                                                                         \end{align}
	Now, applying $P$/$Q$-balancing to \eqref{goal_a} under the assumptions we made in this subsection, the reduced system \eqref{alternative_reduced2_strat} preserves this property, i.e., there exist a positive definite matrix $X_r$, so that \begin{align}\label{stab_preserv}
		A_{N, 11} X_r+ X_r A_{N, 11}^\top +\sum_{i=1}^q N_{i, 11} X_r  N_{i, 11}^\top  <0.                                                                                                                                                                                         \end{align}
	This result was established in \cite{redbendamm} given that $\sigma_r\neq \sigma_{r+1}$, where $\sigma_i$ is the $i$th diagonal entry of $\Sigma$. If $P$-balancing is used instead, \eqref{stab_preserv} basically holds as well \cite{one_sided_BT}. However, generally a further Galerkin projection of the reduced system (not causing an error) is required in order to ensure stability preservation. We illustrated with the following example that stability is not necessarily preserved in  \eqref{goal_a_truncated} given the Stratonovich case.
	\begin{example}\label{example_no_stab}
		Let us fix $x_0=0$, $q=1$ and consider \eqref{goal_a} with
		\begin{align*}
			A= \begin{bmatrix} {-\frac{13}{8}}&{\frac{5}{4}}\\
				{-\frac{5}{4}}&{-2}\end{bmatrix},\quad {B}=C^\top = \begin{bmatrix}  1\\ 0\end{bmatrix} ,\quad N_1= \begin{bmatrix} \frac{3}{2}& -1\\
				1&1\end{bmatrix}
		\end{align*}
		and hence $A_N=\smat -1& 0\\ 0 & -2\srix$. This system is asymptotically mean square stable, since \eqref{stab_FOM} is satisfied. We apply $P$/$Q$-balancing in order to compute ROMs \eqref{goal_a_truncated} and \eqref{alternative_reduced2_strat} for $r=1$ and $H=1/2$. Now, we find that $2A_{N, 11}+N_{1, 11}^2=-0.85926<0$ which is equivalent to \eqref{stab_preserv} in the scalar case. On the other hand, \eqref{goal_a_truncated} is not stable, because $2(A_{11}+0.5N_{1, 11}^2)+N_{1, 11}^2= 0.13825>0$.
	\end{example}
	Example \ref{example_no_stab} shows us that we cannot generally expect a good approximation of \eqref{goal_a} by \eqref{goal_a_truncated} in the Stratonovich setting as the asymptotic behaviour can be contrary.\smallskip
	
	We emphasize this argument further by looking at the error of the approximations if the full model \eqref{goal_a} and the reduced system \eqref{goal_a_truncated} have the same asymptotic behaviour. First, let us note the following. If \eqref{goal_a} is mean square asymptotically stable, then applying $P$- or $P$/$Q$-balancing to this equation ensures the existence of a matrix $\mathcal W$ (depending on the method), so that \begin{align}\label{rep_bound0}
		\sup_{t\in [0, T]}\mathbb E \left\|y(t) - y_r(t) \right\|_2 \leq   \left(\trace\Big(\Sigma_2\mathcal W\Big)\right)^{\frac{1}{2}}  \left\| u\right\|_{L^2_T},                                                                                                                                                                                                                                                                                                                                                                                                                                                                                                                                                                                                                                                                           \end{align}
	where $y_r$ is the output of \eqref{alternative_reduced2_strat}. This was proved in \cite{redmannbenner, one_sided_BT}. Notice that  $\mathcal W$ is independent of the diagonalized Gramian $\Sigma$ and $\Sigma_2$ contains the truncated eigenvalues only, see \eqref{partition_sig}. It is important to mention that \cite{one_sided_BT} just looked at the $P$-balancing case if $C=I$ but \eqref{rep_bound0} holds for general $C$, too. Let us now look at ROM \eqref{goal_a_truncated} and check for a bound like \eqref{rep_bound0}. First of all, we need to assume stability preservation in \eqref{goal_a_truncated} for the existence of a bound. This preservation is not naturally given according to Example \ref{example_no_stab} in contrast to \eqref{alternative_reduced2_strat}.
	\begin{theorem} \label{sig1_bound}
		Given that we consider the Stratonovich setting of $H=1/2$.
		Let system \eqref{goal_a} with output $y$ and $x_0=0$ be mean square asymptotically stable. Moreover, suppose that \eqref{goal_a_truncated} with output $y_r$ and $x_{0, r} =0$ preserves this stability. In case \eqref{goal_a_truncated} is based on either $P$-balancing or $P$/$Q$-balancing according to Definition \ref{defn_balancing_pro} ($iii$), we have \begin{align}\label{rep_bound}
			\sup_{t\in [0, T]}\mathbb E \left\|y(t) - y_r(t) \right\|_2 \leq
			\left(\trace\Big(\Sigma_1(\hat Q_1^\top-Q_r)\Delta_{N, 11}\Big)+ \trace\Big(\Sigma_2\mathcal W\Big)\right)^{\frac{1}{2}}  \left\| u\right\|_{L^2_T},                                                                                                                                                                                                                                                                                                                                                                                                                                                                                                                                                                                                                                                                           \end{align}
		where $\mathcal W:=C_2^\top C_2+2{A}^\top_{N, 12} \hat Q_2 + \sum_{i=1}^{q} {N}^\top_{i, 12} \Big(2\hat Q \smat {N}_{i, 12}\\
		{N}_{i, 22}\srix - Q_r  {N}_{i, 12}\Big)$.
		The above matrices result from the partition \eqref{partition} of the balanced realization \eqref{goal_a_trans} of \eqref{goal_a} and $\tilde A_N= \smat{A}_{N, 11}&{A}_{N, 12}\\
		{A}_{N, 21}&{A}_{N, 22}\srix$, where $\tilde A_N = \tilde A + \frac{1}{2}\sum_{i=1}^q \tilde N_i^2$. Moreover, we set $\Delta_{N, 11}=\sum_{i=1}^q N_{i, 12}N_{i, 21}$ and assume that
		$\hat Q=\smat {\hat Q_1}& {\hat Q_2}\srix$ and $Q_r$ are the unique solutions to \begin{align}\label{yequation}
			(A_{N, 11}-\frac{1}{2}\Delta_{N, 11})^\top \hat Q+ \hat Q \tilde A_N +\sum_{i=1}^{q}  {N}^\top_{i, 11} \hat Q \tilde N_{i} &= -C_1^\top \tilde C, \\ \label{Qr_eq}
			(A_{N, 11}-\frac{1}{2}\Delta_{N, 11})^\top Q_r+ Q_r (A_{N, 11}-\frac{1}{2}\Delta_{N, 11})+\sum_{i=1}^{q} N_{i, 11}^\top Q_r N_{i, 11}&=-C_1^\top C_1.
		\end{align}
		The bound in \eqref{rep_bound} further involves the matrix $
		\Sigma= \smat{\Sigma}_{1}& \\
		&{\Sigma}_{2}\srix$ of either eigenvalues of $P$ ($P$-balancing) or square roots of eigenvalues of $PQ$ ($P$/$Q$-balancing). In particular, $\Sigma_2$ represents the  truncated eigenvalues of the system.
	\end{theorem}
	\begin{proof}
		We have to compare the outputs of \eqref{goal_a_trans} and \eqref{goal_a_truncated}. This is the same like calculating the error between the corresponding Ito versions of these systems. In the Ito equation of  \eqref{goal_a_trans},  $\tilde A$ is replaced by $\tilde A_{N}$ and the Ito form of \eqref{goal_a_truncated} involves $A_{11}+\frac{1}{2}\sum_{i=1}^q N_{i, 11}^2=A_{N, 11}-\frac{1}{2}\Delta_{N, 11}$ instead of $A_{11}$. Since either $P$-balancing or $P$/$Q$-balancing is used, we know that at least one of the Gramians is diagonal, i.e., $P=\Sigma$ (see Proposition \ref{prop_diag_gram}). Since we are in the case of $H=1/2$, we also know the relation to Lyapunov equations by Section \ref{sec_comp_gram}, so that we obtain \begin{align}\label{full_gram2}
			\tilde A_N \Sigma + \Sigma \tilde A_N^\top+ \sum_{i=1}^{q} \tilde N_{i} \Sigma \tilde N_{i}^\top &= -\tilde B \tilde B^\top.\end{align}
		In the Ito setting, an error bound has been established in \cite{redmannbenner}. Applying this result yields
		\begin{align}\label{firstEB}
			\sup_{t\in [0, T]}\mathbb E \left\|y(t) - y_r(t) \right\|_2 \leq  \left(\trace(\tilde C \Sigma \tilde C^\top) +   \trace(C_1 P_r C_1^\top) - 2 \trace(\tilde C \hat P C_1^\top)\right)^{\frac{1}{2}}  \left\| u\right\|_{L^2_T}.
		\end{align}
		The reduced system Gramian $P_r$ as well as the mixed Gramian $\hat P$ exist due to the assumption that stability is preserved in the reduced system. They can be defined as the unique solutions of \begin{align} \label{red_Gram}
			(A_{N, 11}-\frac{1}{2}\Delta_{N, 11})P_r+ P_r (A_{N, 11}-\frac{1}{2}\Delta_{N, 11})^\top+\sum_{i=1}^{q} N_{i, 11} P_r N_{i, 11}^\top &=-B_1B_1^\top,
			\\  \label{mixed_gram2}
			\tilde A_N \hat P + \hat P (A_{N, 11}-\frac{1}{2}\Delta_{N, 11})^\top+\sum_{i=1}^{q} \tilde N_{i} \hat P {N}_{i, 11}^\top &= -\tilde B B_1^\top.
		\end{align}
		Using the partitions of $\tilde A_N$ and the other matrices in \eqref{partition}, we evaluate the first $r$ columns of \eqref{full_gram2} to obtain \begin{align}\label{rel_bla}
			- \tilde B B_1^\top &=    \tilde A_N \smat{\Sigma}_{1}\\
			0 \srix + \Sigma \smat{A}^\top_{N, 11}\\{A}^\top_{N, 12}\srix+ \sum_{i=1}^{q} \tilde N_{i} \Sigma \smat{N}^\top_{i, 11}\\{N}^\top_{i, 12}\srix \\    \nonumber
			&=  \smat{A}_{N, 11}\\ {A}_{N, 21}\srix \Sigma_{1} + \smat \Sigma_1 {A}^\top_{N, 11}\\ \Sigma_2 {A}^\top_{N, 12}\srix+ \sum_{i=1}^{q} \left(\smat{N}_{i, 11}\\
			{N}_{i, 21}\srix  \Sigma_1 {N}^\top_{i, 11} +\smat {N}_{i, 12}\\
			{N}_{i, 22}\srix \Sigma_2 {N}^\top_{i, 12}\right).                                                                                                                         \end{align}
		Using the properties of the trace, we find the relation  $\trace(\tilde C \hat P C_1^\top) = \trace(\hat Q \tilde B B_1^\top)$ between the mixed Gramians satisfying \eqref{yequation} and \eqref{mixed_gram2}. We insert \eqref{rel_bla} into this relation giving us \begin{align*}
			-\trace(\tilde C \hat P C_1^\top)  &=  \trace\left(\hat Q\left[ \smat{A}_{N, 11}\\ {A}_{N, 21}\srix \Sigma_{1} + \smat \Sigma_1 {A}^\top_{N, 11}\\ \Sigma_2 {A}^\top_{N, 12}\srix+ \sum_{i=1}^{q} \left(\smat{N}_{i, 11}\\
			{N}_{i, 21}\srix  \Sigma_1 {N}^\top_{i, 11} +\smat {N}_{i, 12}\\
			{N}_{i, 22}\srix \Sigma_2 {N}^\top_{i, 12}\right)\right]\right)  \\
			&=   \trace\left(\Sigma_1\left[\hat Q \smat{A}_{N, 11}\\ {A}_{N, 21}\srix+({A}_{N, 11}-\frac{1}{2}\Delta_{N, 11})^\top \hat Q_1 + \sum_{i=1}^{q}  {N}^\top_{i, 11} \hat Q \smat{N}_{i, 11}\\
			{N}_{i, 21}\srix \right]\right) \\
			&\quad +   \frac{1}{2} \trace\left(\Sigma_1\Delta_{N, 11}^\top \hat Q_1 \right)+\trace\left(\Sigma_2\left[{A}^\top_{N, 12} \hat Q_2 + \sum_{i=1}^{q} {N}^\top_{i, 12} \hat Q \smat {N}_{i, 12}\\
			{N}_{i, 22}\srix \right]\right).
		\end{align*}
		The first $r$ columns of \eqref{yequation} give us $\hat Q \smat{A}_{N, 11}\\ {A}_{N, 21}\srix+({A}_{N, 11}-\frac{1}{2}\Delta_{N, 11})^\top \hat Q_1 + \sum_{i=1}^{q}  {N}^\top_{i, 11} \hat Q \smat{N}_{i, 11}\\
		{N}_{i, 21}\srix= -C_1^\top C_1$ and hence
		\begin{align*}
			-\trace(\tilde C \hat P C_1^\top) = -\trace(C_1\Sigma_1C_1^\top)+  \frac{1}{2} \trace\left(\Sigma_1\Delta_{N, 11}^\top \hat Q_1 \right)+\trace\left(\Sigma_2\left[{A}^\top_{N, 12} \hat Q_2 + \sum_{i=1}^{q} {N}^\top_{i, 12} \hat Q \smat {N}_{i, 12}\\
			{N}_{i, 22}\srix \right]\right).
		\end{align*}
		We exploit this for the bound in \eqref{firstEB} and further find that $\trace(\tilde C \Sigma \tilde C^\top) =\trace(C_1 \Sigma_1 C_1^\top) +\trace(C_2 \Sigma_2 C_2^\top) $. Thus, we have
		\begin{align}\label{second_last_est}
			&\trace(\tilde C \Sigma \tilde C^\top) +   \trace(C_1 P_r C_1^\top) - 2 \trace(\tilde C \hat P C_1^\top)\\ \nonumber
			&= \trace(C_1(P_r-\Sigma_1)C_1^\top)+\trace\left(\Sigma_1\Delta_{N, 11}^\top \hat Q_1 \right)+\trace\left(\Sigma_2\left[C_2^\top C_2+2{A}^\top_{N, 12} \hat Q_2 + 2\sum_{i=1}^{q} {N}^\top_{i, 12} \hat Q \smat {N}_{i, 12}\\
			{N}_{i, 22}\srix \right]\right).
		\end{align}
		Now, we analyze $P_r-\Sigma_1$. The left upper $r\times r$ block of \eqref{full_gram2}  fulfills \begin{align*}
			& (A_{N, 11}-\frac{1}{2}\Delta_{N, 11})\Sigma_1+ \Sigma_1 (A_{N, 11}-\frac{1}{2}\Delta_{N, 11})^\top+\sum_{i=1}^{q} N_{i, 11} \Sigma_1 N_{i, 11}^\top \\
			&=-B_1B_1^\top-\sum_{i=1}^q N_{i, 12}\Sigma_2 N_{i, 12}^\top-\frac{1}{2}\Delta_{N, 11}\Sigma_1- \Sigma_1
			\frac{1}{2}\Delta_{N, 11}^\top.                                                                                                                                                                                               \end{align*}
		Comparing this with \eqref{red_Gram} yields \begin{align*}
			&(A_{N, 11}-\frac{1}{2}\Delta_{N, 11})(P_r-\Sigma_1)+ (P_r-\Sigma_1) (A_{N, 11}-\frac{1}{2}\Delta_{N, 11})^\top+\sum_{i=1}^{q} N_{i, 11} (P_r-\Sigma_1) N_{i, 11}^\top \\
			&=\sum_{i=1}^q N_{i, 12}\Sigma_2 N_{i, 12}^\top+\frac{1}{2}\Delta_{N, 11}\Sigma_1 + \Sigma_1
			\frac{1}{2}\Delta_{N, 11}^\top.
		\end{align*}
		Therefore, using \eqref{Qr_eq}, we obtain that
		\begin{align*}
			&\trace(C_1(P_r-\Sigma_1)C_1^\top)=\trace((P_r-\Sigma_1)C_1^\top C_1)\\
			&=-\trace\Big((P_r-\Sigma_1)[(A_{N, 11}-\frac{1}{2}\Delta_{N, 11})^\top Q_r+ Q_r (A_{N, 11}-\frac{1}{2}\Delta_{N, 11})+\sum_{i=1}^{q} N_{i, 11}^\top Q_r N_{i, 11}]\Big)\\
			&=-\trace\Big([(A_{N, 11}-\frac{1}{2}\Delta_{N, 11})(P_r-\Sigma_1)+(P_r-\Sigma_1)(A_{N, 11}-\frac{1}{2}\Delta_{N, 11})^\top+\sum_{i=1}^{q} N_{i, 11}(P_r-\Sigma_1) N_{i, 11}^\top] Q_r \Big)\\
			&=-\trace\Big([\sum_{i=1}^q N_{i, 12}\Sigma_2 N_{i, 12}^\top+\Delta_{N, 11}\Sigma_1] Q_r \Big)=-\trace\Big([\Sigma_2\sum_{i=1}^q  N_{i, 12}^\top Q_r N_{i, 12}+\Sigma_1 Q_r \Delta_{N, 11}]\Big).
		\end{align*}
		Inserting this into \eqref{second_last_est} concludes the proof.
	\end{proof}
	Even if stability is preserved in \eqref{goal_a_truncated}, we cannot ensure a small error if we only know that $\Sigma_2$ is small. This is the main conclusion from Theorem \ref{sig1_bound} as the bound depends on a potentially very large matrix $\Sigma_1$. This is an indicator that there are cases in which \eqref{goal_a_truncated} might perform poorly. The correction term $\frac{1}{2}\Delta_{N, 11}=\frac{1}{2}\sum_{i=1}^q N_{i, 12}N_{i, 21}$ in \eqref{alternative_reduced2_strat} ensures that the expression in \eqref{rep_bound} that depends on  $\Delta_{N, 11}$ is canceled out. This leads to the bound in \eqref{rep_bound0}. Let us conclude this paper by conducting several numerical experiments.
	
	\section{Numerical results}\label{sec5}

	In this section, the reduced order techniques that are based on balancing and lead to a system like in \eqref{goal_a_truncated} or \eqref{alternative_reduced2_strat} are applied to two examples. In detail, stochastic heat and wave equations driven by fractional Brownian motions with different Hurst parameters $H$ are considered and formally discretized in space. This discretization yields a system of the form \eqref{goal_a} which we reduce concerning the state space dimension. Before we provide details on the model reduction procedure, let us briefly describe the time-discretization that is required here as well. We use an implicit scheme, because spatial discretizations of the underlying stochastic partial differential equations are stiff.
	
	\subsection{Time integration}\label{Tim}
	The stochastic differential equations \eqref{goal_a}, \eqref{goal_a_truncated} and \eqref{alternative_reduced2_strat}  can be numerically solved by employing a variety of general-purpose stochastic numerical schemes (see, e.g., \cite{Hu_Liu_Nualart, kloeden_platen, Neuenkirch} and the references therein). Encountered frequently in practice, stiff differential equations present a difficult challenge for numerical simulation in both deterministic and stochastic systems. Implicit methods are generally found to be more effective than explicit methods for solving stiff problems. The goal of this work is to exploit an implicit numerical method that is well-suited for addressing stiff stochastic differential equations. The stochastic implicit midpoint method will be the subject of our attention throughout the entire numerical section. We refer to \cite{Hong} ($H>1/2$) and \cite{redmann2020runge} ($H= 1/2$) for more detailed consideration on Runge-Kutta methods based on increments of the driver. In particular, the stochastic implicit midpoint method is a Runge-Kutta scheme with Butcher tableau \begin{align*}
		\begin{array}{l|l}
			\frac{1}{2}     &  \frac{1}{2} \\
			\hline
			&  1  \end{array}.
	\end{align*}
	It therefore takes the form
	\begin{equation}\label{mid_point}
		\begin{split}
			x_{k+1}=x_k+\Big[A\Big(\frac{x_k+x_{k+1}}{2}\Big)+Bu\Big(t_k+\frac{\Delta t}{2}\Big)\Big]\Delta t+\sum_{i=1}^q N_i\Big(\frac{x_k+x_{k+1}}{2}\Big)\Delta W^H_{i, k}
		\end{split}
	\end{equation}
	when applying it to \eqref{goal_a}, where $\Delta t$ denotes the time step related to equidistant grid points $t_k$. Moreover, we define $ \Delta W^H_{i, k} =  W^H_i(t_k+1)-W^H_i(t_k)$. The midpoint method converges with almost sure/$L^p$-rate (arbitrary close to) $2H-1/2$ for $H\in[1/2, 1)$.

	\subsection{Dimension reduction for a stochastic heat equation}

	We begin with a modified version of an example studied in \cite{redmannbenner}. In particular, not an Ito equation driven by a Brownian motion is studied. Instead, we
	consider the following Young/Stratonovich stochastic partial differential equation driven by a (scalar) fractional Brownian motion $W^H$ with Hurst parameter $ H\in[1/2,1)$:
	\begin{equation}\label{ex1}
		\begin{split}
			\frac{\partial X(t,\zeta)}{\partial t}&=a \Delta X(t,\zeta)+1_{[\frac{\pi}{4},\frac{3\pi}{4}]^2}(\zeta)u(t)+\gamma e^{-|\zeta_1-\frac{\pi}{2}|-\zeta_2}X(t,\zeta)\circ \frac{\partial W^H(t)}{\partial t},\quad t\in[0,1],\quad \zeta\in[0,\pi]^2,\\
			X(t,\zeta)&=0,\quad t\in[0,1],\quad \zeta\in\partial[0,\pi]^2,\quad \text{and} \quad	X(0,\zeta)=b \cos(\zeta),
		\end{split}
	\end{equation}
	where $ a ,b>0 $, $ \gamma\in\mathbb{R} $ and  a single input meaning that $m=1$. The solution space for a mild solution is supposed to be $ \mathcal H=L^2([0,\pi]^2) $ exploiting that the Dirichlet Laplacian generates a $C_0$-semigroup. The following average temperature is assumed to be the quantity of interest:
	\[Y(t)=\frac{4}{3\pi^2}\int_{[0,\pi]^2\setminus[\frac{\pi}{4},\frac{3\pi}{4}]^2}X(t,\zeta)d\zeta.\]
	Based on the eigenfunctions of the Laplacian, a spectral Galerkin scheme analogous to the method, proposed and explained in \cite{redmannbenner}, is applied to \eqref{ex1}. Roughly speaking, such a discretization is based on an orthogonal projection onto the subspace spanned by the dominant eigenvector of the Laplacian. This results in system \eqref{goal_a} of order $n$ with scalar control and a fixed initial state $x_0$. The detailed structure of the matrices $A, B, N_1$ and $C$ can be found in \cite{redmannbenner}. In the following, we fix $ a=0.2 $, $ b=1$ and set $n = 1024$.
	We investigate two cases. These are $H=0.5$ and $H=0.75$. In the following, we explain the particular dimension reduction techniques for each scenario. \\
	\textbf{Case $\mathbf{H=0.75}:$}\quad
	We have pointed out in Section \ref{sec_comp_gram} that Gramians $P_T$ and $Q_T$ (or their limits $P$ and $Q$) are hard to compute for $H>1/2$, since a link of these matrices to ordinary differential or algebraic equations is unknown. Therefore, we solely consider empirical Gramians discussed in Section \ref{sec_comp_gram} for $H=0.75$. In fact, $\bar P_T$ is available by sampling the solution of \eqref{stan2}, whereas $\bar Q_T$ seems computational much more involved. For that reason, we apply $\bar P_T$-balancing (see Definition \ref{defn_balancing_pro} ($iv$) to system \eqref{goal_a} that obtained from the above heat equation. This results in \eqref{goal_a_trans} which is truncated in order to find the reduced equation \eqref{goal_a_truncated}. Two other related approaches are conducted in this section as well.
	\begin{itemize}
		\item
		We apply the same $\bar P_T$-balancing procedure to subsystems \eqref{stan1} and \eqref{stan2}, i.e., $\bar {P}_{u, T}$-balancing is used for \eqref{stan1} and $\bar{P}_{x_0, T}$-balancing for \eqref{stan2}, compare with \eqref{empirical_sub_gram}. The sum of the resulting reduced order systems is then used to approximate \eqref{goal_a}. For refer to this second ansatz as splitting-based $\bar P_T$-balancing.
		\item Another empirical dimension reduction technique called proper orthogonal decomposition (POD) is available for this setting \cite{Jamshidi}. For this method, the solution space of \eqref{goal_a} is learned using samples. In that context, a snapshot matrix with columns of the form $x(t_i,\omega_j)$ is computed with $i=1, \dots, \mathcal N$ and $j=1, \dots, \mathcal N_s$, where $t_i\in [0, T]$ and $\omega_j\in\Omega$. These samples are potentially based on various initial states  $x_0$ and controls $u$. Notice that snapshot matrices are  computed based on a small set of $x_0$ and $u$ aiming to provide ROMs performing well for a large number of $x_0$ and $u$. We end up with a POD-based reduced system \eqref{reduced}, where the projection matrix $V=W$ consists of vectors associated to large singular values of the snapshot matrix. Instead of using POD for \eqref{goal_a} directly, we apply it to subsystems \eqref{stan1} and \eqref{stan2} and find an approximation for \eqref{goal_a} by the sum of the reduced subsystems. We call this splitting-based POD.
	\end{itemize}
	\textbf{Case $\mathbf{H=0.5}$:}\quad
	Similar techniques are exploited for the Stratonovich setting. However, we have the advantage that $P_T$ and $Q_T$ can be computed from matrix equations, see \eqref{eq_PT} and \eqref{eq_QT}. Still these equations are difficult to solve. Therefore, we use the sampling and variance reduction based schemes proposed in \cite{redmann2022gramian} in order to solve them. Due to the availability of both Gramians, we apply $P_T$/$Q_T$-balancing, see Definition \ref{defn_balancing_pro} ($ii$), instead of the procedure based on diagonalizing $\bar P_T$. However, we truncate differently, i.e., the reduced system \eqref{alternative_reduced2_strat} is used instead due to the drawbacks of \eqref{goal_a_truncated} pointed out in Section \ref{drawback_rom} when $H=0.5$.
	The splitting-based $P_T$/$Q_T$-balancing is defined the same way. It is the technique, where ${P}_{u, T}$/$Q_T$-balancing is conducted for \eqref{stan1} and ${P}_{x_0, T}$/$Q_T$-balancing  is exploited for \eqref{stan2} to obtain reduced systems of the form \eqref{alternative_reduced2_strat} for each subsystem. Again, we use a splitting-based POD scheme according to \cite{Jamshidi} for $H=0.5$.\smallskip
	
	For the discretization in time, the stochastic midpoint method \eqref{mid_point}, stated in Section \ref{Tim}, is employed here, where the number of time steps is $\mathcal N = 100$. Moreover, all empirical objects are calculated based on $\mathcal N_s= 10^3$ samples. The error between the reduced systems and the original model is computed for the control $u(t)=\sqrt{\frac{2}{\pi}}\sin(t) $, where the reduction error is measured by the quantity $ \mathcal{R}_{ E}=\dfrac{\sup_{t\in[0,1]}\mathbb{E}\|y(t)-y_r(t)\|_2}{\sup_{t\in[0,1]}\mathbb{E}\|y(t)\|_2} $.
	
	In the case of $H=0.5$, Figure \ref{fig:1} illustrates that splitting-based $P_T$/$Q_T$-balancing (2. Gramian) and $P_T$/$Q_T$-balancing (1. Gramian) generate very similar results. Both techniques produce notably better outcomes compared to the splitting-based POD method. The worst case errors of the plot are also state in the associated Table \ref{tab:1}.

	\begin{figure}
		\begin{minipage}{75mm}
			\includegraphics[width=\linewidth,height=50mm]{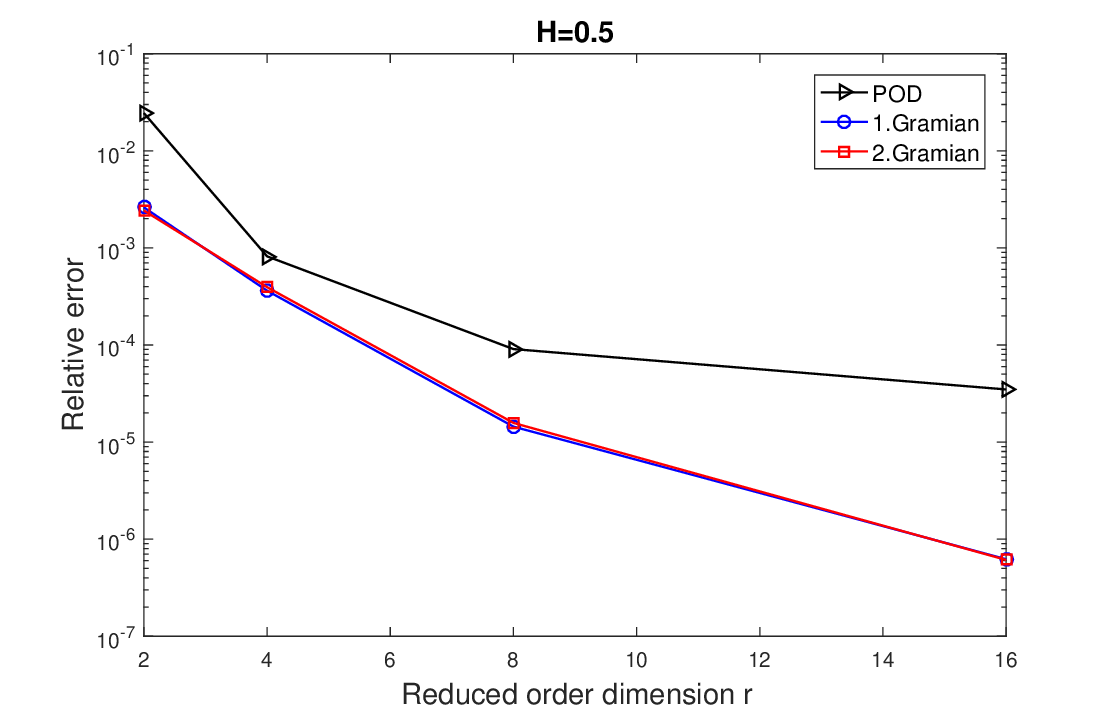}
			\caption{$\mathcal{R}_E $ for three approaches with Hurst parameters $ H=0.5$. }
			\label{fig:1}
		\end{minipage}
		\hfil
		\begin{minipage}{65mm}
			\captionsetup{type=table}
			\caption{$\mathcal{R}_E $ for $ r\in\{2,4,8,16\} $ and $ H=0.5 $.}
			\label{tab:1}
			\begin{tabular}{@{}llll@{}}
				\hline
				$ r $  & POD &  1. Gramian & 2. Gramian  \\
				\hline
				
				2           & $ 2.4471e-02  $ & $ 2.6131e-03  $ & $ 2.4251e-03   $ \\
				
				4         & $ 8.1898e-04  $ & $ 3.6254e-04  $ & $ 3.9410e-04 $ \\
				
				8           & $ 9.0777e-05  $ & $ 1.4427e-05  $ & $1.5756e-05  $ \\
				
				16         & $ 3.4842e-05 $ & $ 6.2128e-07  $&$ 6.1161e-07 $  \\
				
				\hline
			\end{tabular}
		\end{minipage}
	\end{figure}
	
	On the other hand, the Young setting in which we have $H=0.75$ presents a different scenario. Figure \ref{fig:2} demonstrates that splitting-based POD exhibits a better performance compared to splitting-based $\bar{P}_T$-balancing (2. Gramian) and the usual $\bar{P}_T$-balancing (1. Gramian), except when the reduced dimension is $ 16 $. Surprisingly, for $ r=16 $, the 2. Gramian method yields better results compared to POD. It is worth noting that both empirical Gramian methods provide similar outcomes, which is an indicator for a nearly identical  reduction potential for both subsystems \eqref{stan1} and \eqref{stan2}. Note that the error of the plot can be found in Table \ref{tab:2}.

	\begin{figure}
		\begin{minipage}{75mm}
			\includegraphics[width=\linewidth,height=50mm]{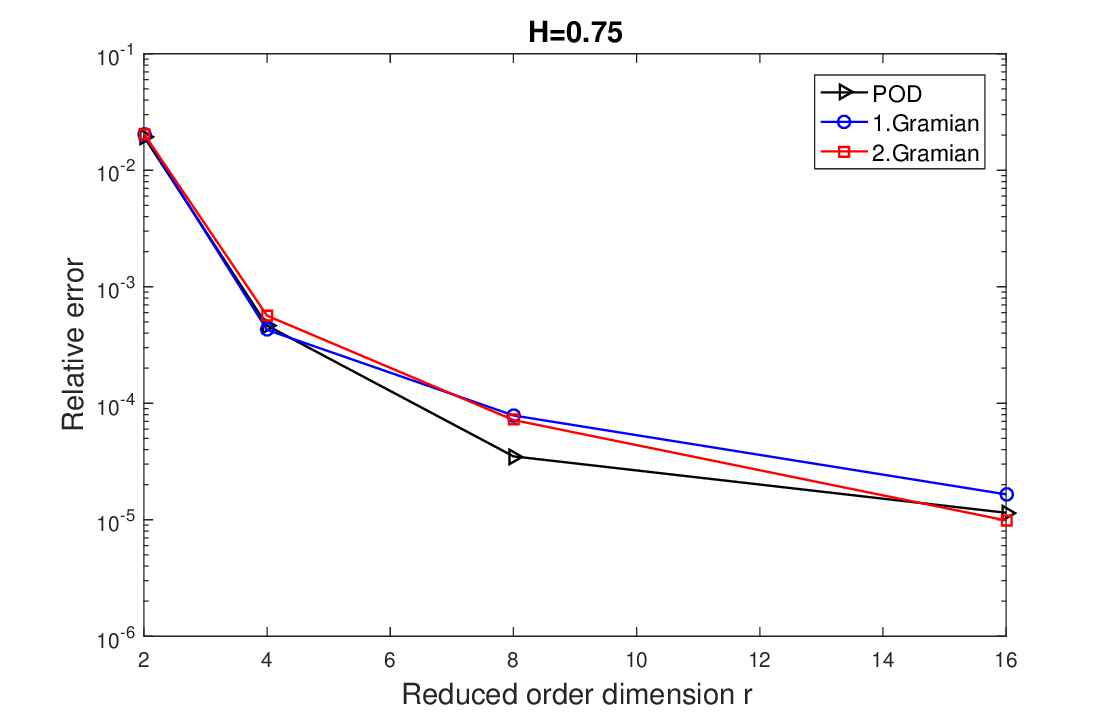}
			\caption{$\mathcal{R}_E $ for three approaches with Hurst parameters $ H=0.75 $. }
			\label{fig:2}
		\end{minipage}
		\hfil
		\begin{minipage}{65mm}
			\captionsetup{type=table}
			\caption{$\mathcal{R}_E $ for $ r\in\{2,4,8,16\} $ and $ H=0.75 $.}
			\label{tab:2}
			\begin{tabular}{@{}llll@{}}
				
				\hline
				$ r $  & POD &  1. Gramian & 2. Gramian  \\
				\hline
				
				2           & $ 1.9428e-02     $  & $ 2.0531e-02   $&$ 2.0543e-02 $ \\
				
				4         & $ 4.6419e-04  $  & $ 4.2626e-04 $&$ 5.6448e-04 $  \\
				
				8           & $ 3.5032e-05  $  & $ 7.8586e-05$ &$ 7.1846e-05 $  \\
				
				16         & $ 1.1479e-05 $  &$ 1.6520e-05   $&$ 9.8581e-06 $  \\
				
				\hline
			\end{tabular}
		\end{minipage}
	\end{figure}
	
	For both, $H=0.5$ and $H=0.75$ an enormous reduction potential can be observed, meaning that small dimensions $r$ lead to accurate approximations. According to Remark \ref{interprete_prop} this is known a-priori by the strong decay of certain eigenvalues associated to the system Gramians, since small eigenvalues indicate variables of low relevance. Given $H=0.75$, Figure \ref{fig:3_1} shows the eigenvalues of $\bar P_T$ (1. Gramian), the sum eigenvalues of $\bar P_{u, T}$ and $\bar P_{x_0, T}$ (2. Gramian) as well as the sum of the singular values corresponding to the POD snapshot matrices of subsystems \eqref{stan1} and \eqref{stan2}. Similar types of algebraic values are considered for $H=0.5$ in Figure \ref{fig:3_2}. Here, square roots of eigenvalues of $P_T Q_T$ (1. Gramian)  or the sum of square roots of eigenvalues of $P_{u, T} Q_T$ and $P_{x_0, T} Q_T$ (2. Gramian) are depicted.
	\begin{figure}
		\begin{minipage}{75mm}
			\includegraphics[width=68mm,height=50mm]{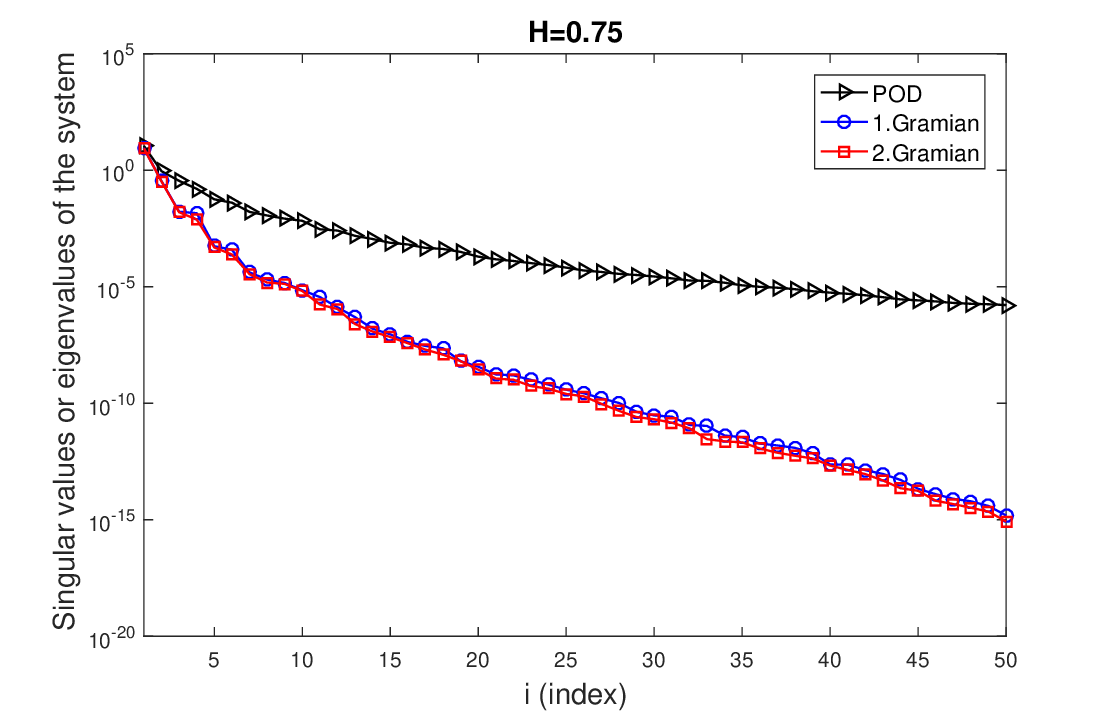}
			\caption{First $50$ POD singular values or eigenvalues associated to $\bar P_T$ for $H=0.75$.}\label{fig:3_1}
		\end{minipage}
		\begin{minipage}{75mm}
			\includegraphics[width=68mm,height=50mm]{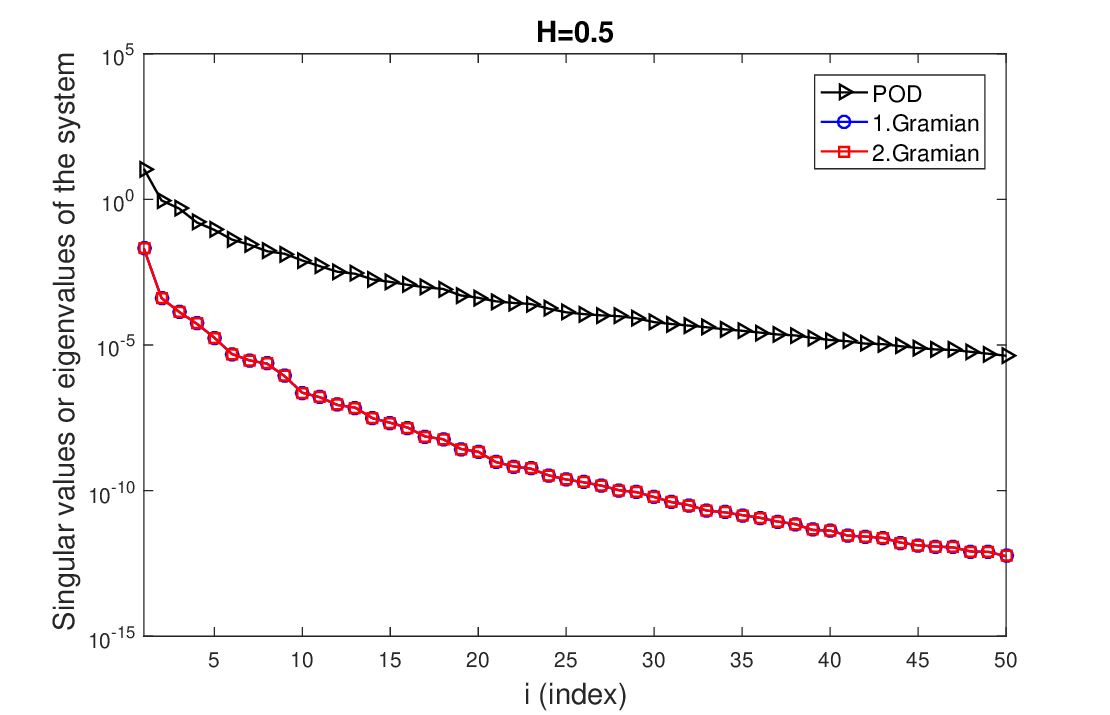}
			\caption{First $50$ POD singular values or eigenvalues associated to $P_T$/$Q_T$ for $H=0.5$.}
			\label{fig:3_2}
		\end{minipage}
	\end{figure}
	The large number of small eigenvalues (or singular values) explains why small errors could be achieved in our simulations.

	\subsection{Dimension reduction for a stochastic wave equation}
	
	We consider the following controlled stochastic partial differential equation which is a modification of the example studied in \cite{redmannbenner2}. In detail, we consider fractional drivers $W^H$ with $H\in [0.5, 1)$ in a Young/Stratonovich setting instead of Ito differential equations driven by a Brownian motion. For $ t\in[0,1]$ and $\zeta\in[0,\pi] $
	\begin{equation}\label{ex2}
		\begin{split}
			&\frac{\partial^2 X(t,\zeta)}{\partial t^2}+a\frac{\partial X(t,\zeta)}{\partial t} =\frac{\partial^2}{\partial \zeta^2} X(t,\zeta)+ e^{-|\zeta-\frac{\pi}{2}|}u(t)+2 e^{-|\zeta-\frac{\pi}{2}|}X(t,\zeta)\circ \frac{\partial W^H(t)}{\partial t},\\
			&X(t,0)=0=X(t,\pi),\quad t\in[0,1],\quad X(0,\zeta)\equiv 0, \quad \frac{\partial}{\partial t} X(t, \zeta)\big\vert_{t=0}=b \cos(\zeta)
		\end{split}
	\end{equation}
	is investigated and the output equation is
	\[Y(t)=\frac{1}{2\epsilon}\left( \int_{\frac{\pi}{2}-\epsilon}^{\frac{\pi}{2}+\epsilon}X(t,\zeta)d\zeta\quad \int_{\frac{\pi}{2}-\epsilon}^{\frac{\pi}{2}+\epsilon}\frac{\partial }{\partial t}X(t,\zeta)d\zeta\right)^\top,\]
	so that both the position and velocity of the middle of the string are observed. Moreover, $ a,b>0 $ and $ \epsilon>0 $. Again the solution of \eqref{ex2} shall be in the mild sense (after transformation into a first order equation), where $X(t,\cdot)\in H_0^1([0, \pi])$ and $\frac{\partial }{\partial t}X(t,\cdot)\in L^2([0, \pi])$. Formally discretizing \eqref{ex2} like in \cite{redmannbenner2}, the spectral Galerkin-based system is given by a model of the form \eqref{goal_a} with $q=1$. We refer to \cite{redmannbenner2} for the details on the matrices of this system. In our simulations, we assume $ b=1 $ and $ a=2 $. Further, the sizes of spatial and time discretization are $n=1000$ and $\mathcal N=100$, respectively. In this example, we consider the same scenario as we did in the first example \eqref{ex1} which means that we calculate a splitting-based POD ROM using snapshots of subsystems \eqref{stan1} and \eqref{stan2} for some $x_0$,  controls $u$ and a low number of samples $\mathcal N_s$.   Moreover, (splitting-based) $\bar P_T$-based balancing is applied to the wave equation given $H=0.75$. If $H=0.5$, empirical Gramians are replaced by exact pairs of Gramians, meaning that (splitting-based) $P_T$/$Q_T$-based balancing is exploited. The results are shown in Figures \ref{fig:3} and \ref{fig:4} for $u(t)=\sqrt{\frac{2}{\pi}}\sin(t)$.

	Based on our observations, we find that the splitting-based $P_T$/$Q_T$-based balancing (2. Gramian) method outperforms the $P_T$/$Q_T$-based balancing (1. Gramian) method for both cases when $H=0.75$ and $H=0.5$. Additionally, the splitting-based POD performs best for $H=0.75$ and worst for $H=0.5$. The results are again presented in Tables \ref{tab:3} and \ref{tab:4}, where the exact numbers are shown.
	
	Interestingly, for both the heat and the wave equation, splitting-based POD performs best in the Young setting ($H=0.75$), but worst in the Stratonovich case ($H=0.5$).
	\begin{figure}
		\begin{minipage}{75mm}
			\includegraphics[width=\linewidth,height=50mm]{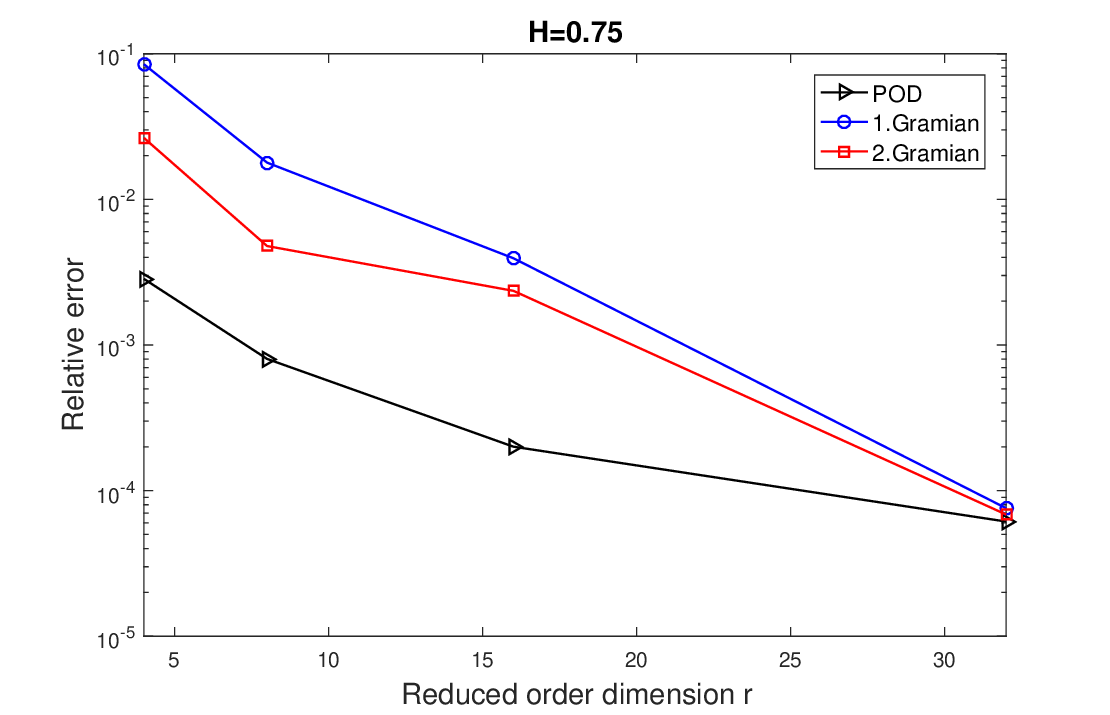}
			\caption{$\mathcal{R}_E $ for three approaches with Hurst parameters $ H=0.75$. }
			\label{fig:3}
		\end{minipage}
		\hfil
		\begin{minipage}{65mm}
			\captionsetup{type=table}
			\caption{$\mathcal{R}_E $ for $ r\in\{4,8,16,32\} $ and $ H=0.75 $.}
			\label{tab:3}
			\begin{tabular}{@{}llll@{}}
				\hline
				$ r $   & POD &  1. Gramian & 2. Gramian  \\
				\hline
				
				4           & $ 2.8447e-03   $ & $ 8.4704e-02   $ & $ 2.6423e-02  $ \\
				
				8         & $ 8.0259e-04  $ & $ 1.7882e-02 $ & $ 4.7821e-03 $ \\
				
				16           & $ 2.0032e-04  $ & $ 3.9414e-03   $ & $ 2.3544e-03  $ \\
				
				32         & $ 6.1316e-05 $ & $ 7.5687e-05   $&$ 6.8516e-05 $  \\
				
				\hline
			\end{tabular}
		\end{minipage}
	\end{figure}
	
	\begin{figure}
		\begin{minipage}{75mm}
			\includegraphics[width=\linewidth,height=50mm]{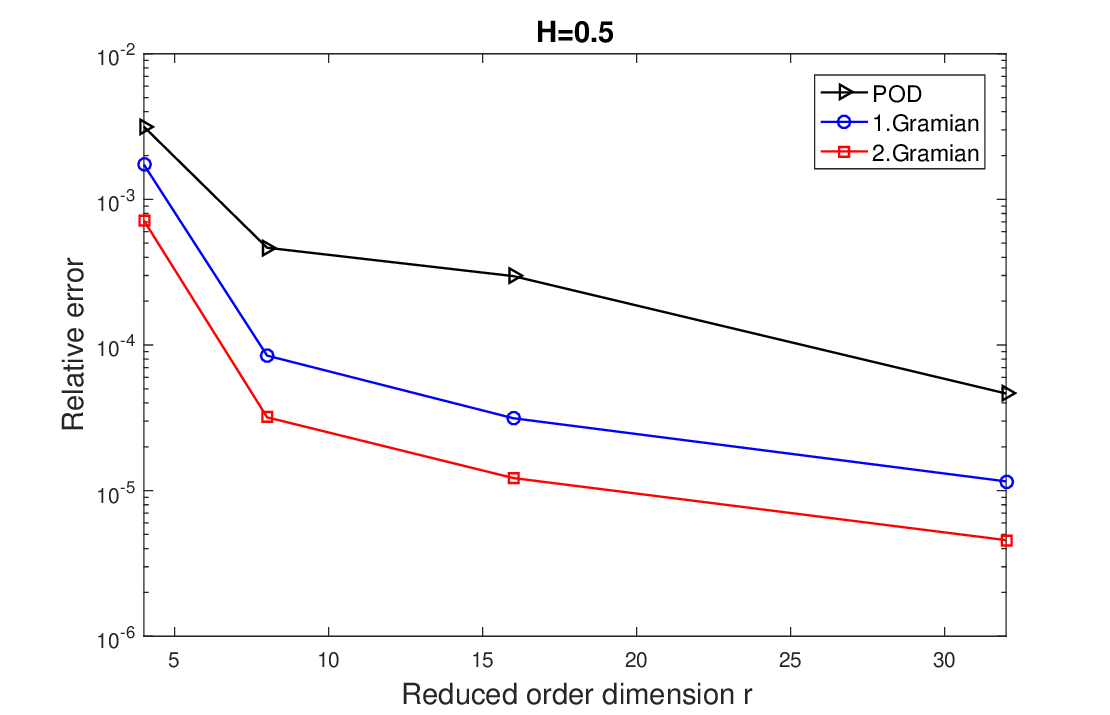}
			\caption{$\mathcal{R}_E $ for three approaches with Hurst parameters $ H=0.5 $. }
			\label{fig:4}
		\end{minipage}
		\hfil
		\begin{minipage}{65mm}
			\captionsetup{type=table}
			\caption{$\mathcal{R}_E $ for $ r\in\{4,8,16,32\} $ and $ H=0.5 $.}
			\label{tab:4}
			\begin{tabular}{@{}llll@{}}
				\hline
				$ r $   & POD &  1. Gramian & 2. Gramian  \\
				\hline
				
				4           & $ 3.1540e-03     $ & $ 1.7312e-03   $ & $ 7.1584e-04   $ \\
				
				8         & $ 4.6545e-04  $ & $ 8.4544e-05  $ & $ 3.1884e-05 $ \\
				
				16           & $ 2.9716e-04 $ & $ 3.1405e-05 $ & $ 1.2200e-05  $ \\
				
				32         & $ 4.6438e-05 $ & $ 1.1572e-05   $&$ 4.5707e-06 $  \\
				
				\hline
			\end{tabular}
		\end{minipage}
	\end{figure}

	\begin{figure}
		\begin{minipage}{75mm}
			\includegraphics[width=68mm,height=50mm]{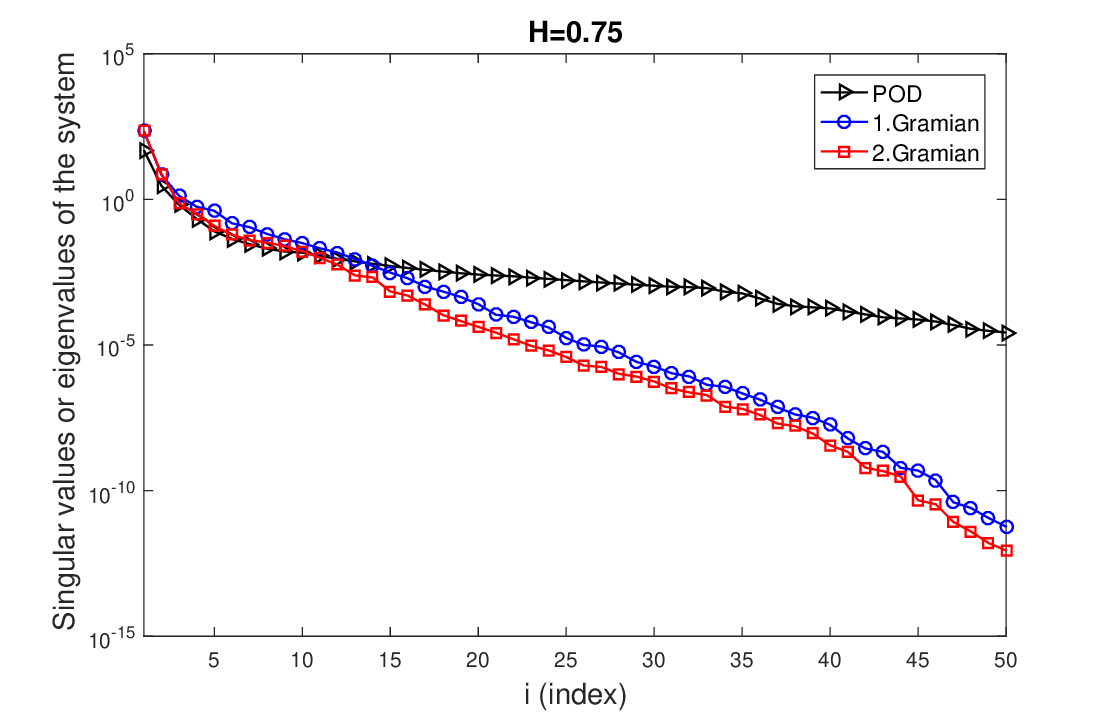}
			\caption{First $50$ POD singular values or eigenvalues associated to $\bar P_T$ for $H=0.75$.}\label{fig:4_1}
		\end{minipage}
		\begin{minipage}{75mm}
			\includegraphics[width=68mm,height=50mm]{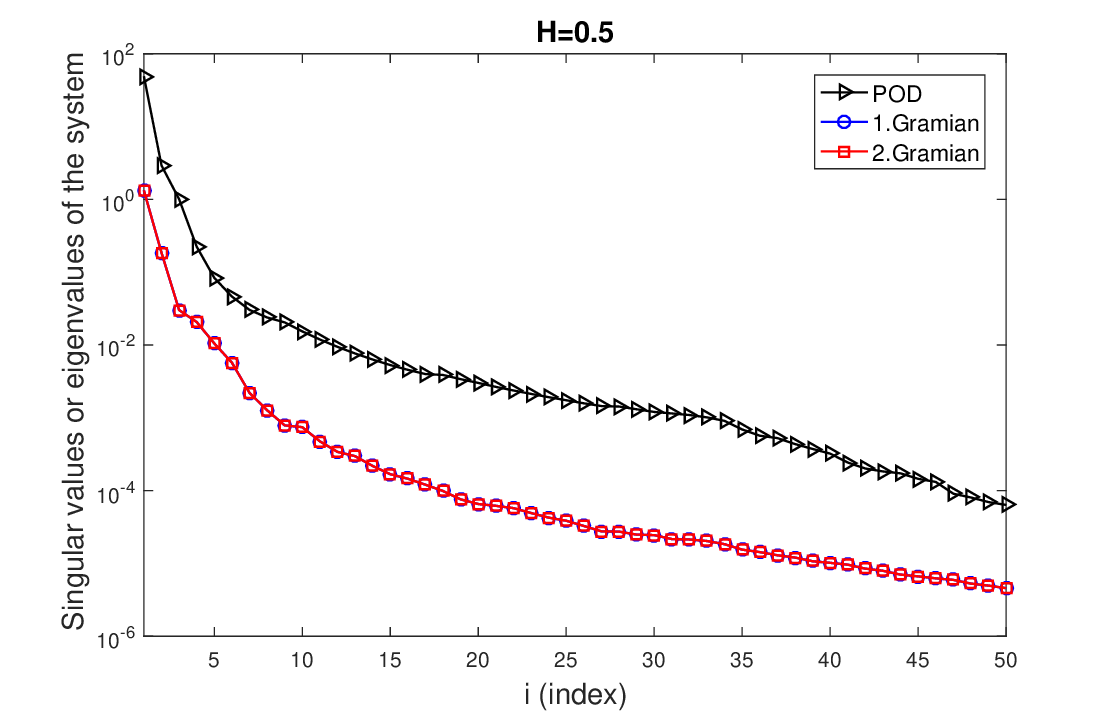}
			\caption{First $50$ POD singular values or eigenvalues associated to $P_T$/$Q_T$ for $H=0.5$.}
			\label{fig:4_2}
		\end{minipage}
	\end{figure}
	
Analogous to Figures \ref{fig:3_1} and \ref{fig:3_2}, Figures \ref{fig:4_1} and \ref{fig:4_2} illustrate the eigenvalues of approximated or exact Gramians as well as the sum of singular values corresponding to the POD snapshot matrices.

	\vspace{\baselineskip}
	%% The style of the following references should be used in all documents.
	
	\bibliographystyle{plain}
	\bibliography{refererences}
	
\end{document}